\documentclass[preprint,12pt,3p]{elsarticle-no-natbib}

\advance\textheight by 0.7truein
\advance\topmargin by -0.3truein
\advance\footskip by 0.2truein

\usepackage{amssymb}
\usepackage{amsthm}

\usepackage[backend=bibtex,%
  style=numeric,%
  maxbibnames=99,%
  bibencoding=ascii,%
  eprint=false,isbn=false,doi=false]{biblatex}

\renewbibmacro{in:}{%
   \ifentrytype{article}{}{\printtext{\bibstring{in}\intitlepunct}}}
  
\usepackage{appendix}
\usepackage{amsmath}
\usepackage{amsfonts}
\usepackage{graphicx}
\usepackage{caption}
\usepackage[labelformat=simple]{subcaption}

\usepackage{color}
\usepackage[shortlabels]{enumitem}
\usepackage{pdfpages}
\usepackage{harpoon}

\newtheorem{theorem}{Theorem}[section]
\newtheorem{proposition}[theorem]{Proposition}
\newtheorem{lemma}[theorem]{Lemma}
\newtheorem{corollary}[theorem]{Corollary}
\newtheorem*{theorem-densemain}{Theorem \ref{densemain}}
\theoremstyle{definition}
\newtheorem{definition}[theorem]{Definition}

\newtheorem{observation}[theorem]{Observation}

\newtheorem{example}[theorem]{Example}

\theoremstyle{definition}

\theoremstyle{remark}

\numberwithin{equation}{section}
\newtheoremstyle{bracketed}%
	{}% space above
	{}% space below
	{}% body font
	{}% indent amount
	{\bfseries}% thm head font
	{}% punc after thm head
	{ }% space after thm head
	{(\thmname{#1}\thmnumber{#2})}% thm head spec
\theoremstyle{bracketed}
\newtheorem{argstep}[theorem]{}
	{\vskip\topsep\noindent\textsc{#1.\/}}%
	{\par\vskip\topsep}%

\definecolor{dkgreen}{rgb}{0,0.7,0.1}
\long\def\ignore#1{}
\def\cA{\mathcal{A}}
\def\cC{\mathcal{C}}

\def\cT{\mathcal{T}}

\def\cW{\mathcal{W}}
\def\cZ{\mathcal{Z}}
\def\mZ{\mathbb{Z}}
\let\al\alpha
\let\be\beta
\let\ga\gamma
\let\de\delta

\let\ka\kappa
\let\la\lambda
\def\set#1{\{#1\}}% set of elements
\def\hf{{\textstyle\frac12}}% 1/2, used frequently in dense graphs paper
\def\iv{^{-1}}% inverse
\def\mmod{{\fam0\ mod\ }}% mod with normal text spacing
\def\deg{{\fam0 deg}}% degree
\def\mw{\!\cdot\!}% merge walks; \cdot has too much space
\def\ug#1{#1_{\fam0 u}}% underlying undirected graph
\def\itr{{\fam0 Tr}}% inner transition graph - now just transition graph
\let\mate\omega% half-edge mate
\let\incv\psi% vertex to which half-edge/arc incident
\def\sbmax{_{\fam0 max}}% subscript max (roman), used for nonoriemtable results
\let\DP\Pi% directed chain of digons with loops at ends
\let\dip\Delta% dipole digraph
\let\ddc\Gamma% cycle of directed digons
\def\bC{\overline{C}}
\def\bl{\overline{\ell}}
\def\bg{\overline{g}}
\def\bh{\overline{h}}
\def\bv{\overline{v}}
\def\ot{
   \overset{\rightharpoonup}%
	{\phantom{\vrule depth0pt height6pt width7pt}}\hskip-8.5pt{T}}
\def\prof(#1,#2){\{{#1},{#2}\}}
\def\jprof(#1;#2){(\prof(#1),\prof(#2))}

\def\siran#1{\v{S}ir\'{a}\v{n}}% e.g. use as "\siran."
\def\skoviera#1{\v{S}koviera}
\def\fijavz#1{Fijav\v{z}}
\def\nebesky#1{Nebesk\'{y}}
\def\yavorskii#1{Javors'ki\u{\i}}% Use name starting with J to be consistent w MathSciNet, also fixed bib

\def\realline{\hbox to \hsize}% plain TeX \line
\journal{To be decided}
\makeatletter
\def\ps@pprintTitleNoSubTo{%
     \let\@oddhead\@empty
     \let\@evenhead\@empty
     \def\@oddfoot{\footnotesize\itshape\hfill\today}%
     \let\@evenfoot\@oddfoot}
\let\ps@pprintTitle\ps@pprintTitleNoSubTo% uncomment for no "Sub to ..."
\makeatother

\begin{document}

\begin{frontmatter}

\title{Bi-eulerian embeddings of graphs and digraphs}

\author[label1, label3]{M. N. Ellingham\corref{cor1}}

\address[label1]{Department of Mathematics, 1326 Stevenson Center,
Vanderbilt University,
Nashville, Tennessee 37240}
\fntext[label3]{Supported by Simons Foundation awards 429625 and MP-TSM-00002760.}

\ead{mark.ellingham@vanderbilt.edu} 
%\ead[url]{https://math.vanderbilt.edu/ellingmn/}

\author[label2]{Joanna A. Ellis-Monaghan}

\address[label2]{Korteweg-de Vries Institute for Mathematics, University of Amsterdam, Science Park 105-107, 1098 XH Amsterdam, the Netherlands}
\ead{jellismonaghan@gmail.com}
%\ead[url]{https://sites.google.com/site/joellismonaghan/}

\begin{abstract}
In 1965 Edmonds showed that every eulerian graph has a \emph{bi-eulerian} embedding, i.e., an embedding with exactly two faces, each bounded by an euler circuit.
We refine this result by giving conditions for a graph to have a bi-eulerian embedding that is specifically orientable or nonorientable.
We give connections to the maximum genus problem for \emph{directed embeddings} of digraphs, in which every face is bounded by a directed circuit.
Given an eulerian digraph $D$ with all vertices of degree $2$ mod $4$ and a directed euler circuit $T$ of $D$, we show that $D$ has an orientable bi-eulerian directed embedding with one of the faces bounded by $T$; this is a maximum genus directed embedding.
This result also holds when $D$ has exactly two vertices of degree $0$ mod $4$, provided they are interlaced by $T$.
More generally, if $D$ has $\ell$ vertices of degree $0$ mod $4$, we can find an orientable directed embedding with a face bounded by $T$ and with at most $\ell+1$ other faces.
We show that given an eulerian graph $G$ and a circuit decomposition $\cC$ of $G$, there is an nonorientable embedding of $G$ with the elements of $\cC$ bounding faces and with one additional face bounded by an euler circuit, unless every block of $G$ is a cycle and $\cC$ is the collection of cycles of $G$.  In particular, every eulerian graph that is not edgeless or a cycle has a nonorientable bi-eulerian embedding with a given euler circuit $T$ bounding one of the faces.
Polynomial-time algorithms giving the specified embeddings are implicit in our proofs.
\end{abstract}


\end{frontmatter}

% \linenumbers

\section{Introduction}\label{sec:intro}

A large segment of topological graph theory is concerned with the existence or optimization of graph embeddings subject to restrictions.   Often the goal is a cellular embedding of a graph in a surface of either minimum or maximum genus (having a maximum or minimum number of faces, respectively).
Other frequent objectives include finding embeddings with restricted faces, for example, where all faces are bounded by $4$-cycles or by hamilton cycles. 
Orientability is often a significant consideration.

Here we focus on two related types of embedding: graph embeddings with faces bounded by euler circuits, and maximum genus directed embeddings of digraphs.  A \emph{directed embedding} of a digraph is an embedding where all faces are bounded by directed walks.  A connected digraph with a directed embedding must be eulerian, so our results always involve eulerian graphs or digraphs.

In the orientable setting, the two types of embedding are closely related, and \emph{bi-eulerian} embeddings, that is, embeddings with exactly two faces each bounded by an euler circuit, are particularly important.  A bi-eulerian directed embedding of a digraph is necessarily orientable and of maximum genus, and an orientable bi-eulerian embedding of an undirected graph can have its edges oriented to give a directed embedding. See Figure \ref{BEembedcirc} for an example of a bi-eulerian direted embedding, represented as a ribbon graph.

Our main result in the orientable setting is that, given an euler circuit $T$ in an eulerian digraph $D$ (or an eulerian graph $G$) having all vertices of degree $2$ mod $4$, there is always an orientable bi-eulerian embedding of $D$ (or $G$) with one of the faces bounded by $T$. If there are $\ell$ vertices of degree $0$ mod $4$, there is still a maximum genus embedding with one face bounded by $T$, and no more than $\ell + 1$ additional faces.  We further show that an orientable bi-eulerian embedding exists when there are two vertices of degree $0$ mod $4$, provided that these two vertices are interlaced by $T$, but not necessarily otherwise.  We then extend this to some sufficient conditions under which there is an orientable bi-eulerian embedding when there are many vertices of degree $0$ mod $4$.
Polynomial-time algorithms giving the specified embeddings are implicit in our proofs. 
The $2$ mod $4$ degree condition plays a central role in the theory, as there is an infinite family of configurations with vertices of degree $4$ whose presence in a graph prevents it from having an orientable bi-eulerian embedding.

As is often the case, the nonorientable setting is easier than the orientable setting.  Leveraging preexisting theory, we show that all nontrivial eulerian graphs except cycles have bi-eulerian embeddings in nonorientable surfaces.  In most cases a given euler circuit, or even a given circuit decomposition, can be completed to a nonorientable embedding with the addition of one additional face bounded by an euler circuit.  In the nonorientable setting, directed embeddings cannot have euler circuit faces, so we discuss embeddings of undirected graphs with euler circuit faces in Section \ref{sec:no-eulf}, separately from maximum nonorientable genus directed embeddings in Section \ref{sec:no-diremb}.

Our motivation to consider embeddings with euler circuit faces is threefold.  Firstly, the question initially arose from DNA self-assembly applications. Secondly, these embeddings extend a body of prior work beginning with Edmonds \cite{Edm65} in 1965, which overlaps with work on euler circuits compatible with transition systems going back to Kotzig \cite{Kot68} in 1968.   Finally, embeddings with euler circuit faces are related to other combinatorial objects and special types of embeddings.

\begin{figure}[h!]
\centering
\includegraphics[width=0.8
 \textwidth]{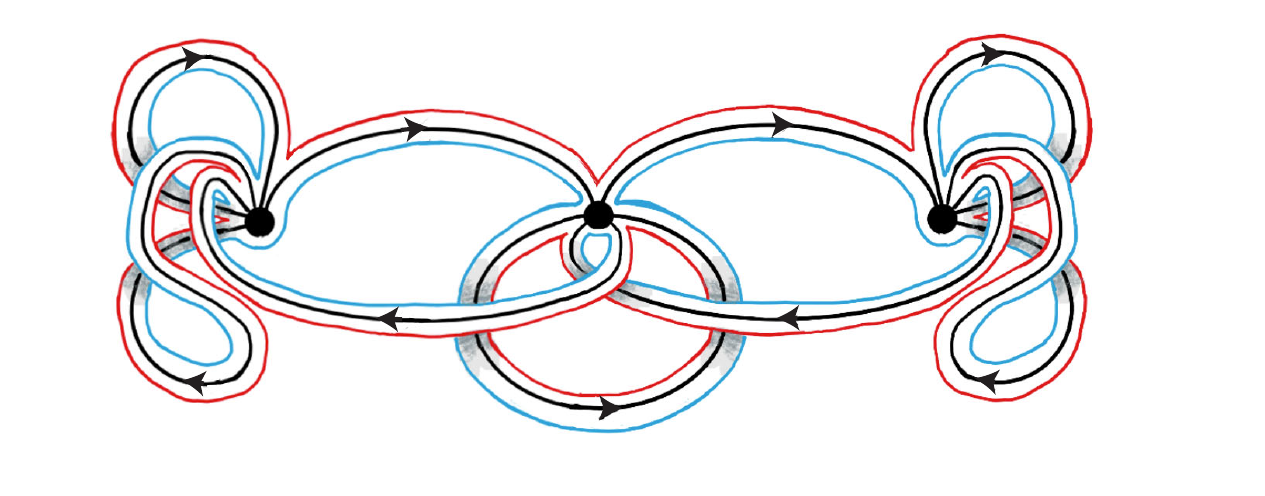}

\caption{A bi-eulerian directed embedding, represented as a ribbon graph.}
\label{BEembedcirc}
\end{figure}  

In \cite{EE-M19} the authors investigated \emph{edge-outer embeddings}, which are orientable cellular embeddings with a special \emph{outer} face whose facial walk uses every edge at least once. Facial walks in edge-outer embeddings model the `scaffolding strands' used in the assembly of wireframe DNA structures, and also the `reporter strands' used to report solutions in DNA computing (see Jonoska, Seeman and Wu \cite{JSW09}). While we showed that finding an edge-outer embedding where the outer face has minimum length is in general NP-hard, this is easy for eulerian graphs. It follows from results on the existence of relative embeddings \cite{SS87} that every eulerian graph has an optimal edge-outer embedding, where the outer face is bounded by an euler circuit.  Furthermore, that euler circuit can be specified in advance.  However, this result gives no control over the number or sizes of the remaining faces.  This leads naturally to the question of additionally minimizing the total number of faces, i.e., finding a maximum genus embedding subject to the initially specified euler circuit.  A bi-eulerian embedding would provide an optimal solution.

The second motivation is that the problem of constructing embeddings with euler circuit faces in various settings forms a long-standing challenge in the field of topological graph theory, dating back to the following 1965 existence result of Edmonds. 

\begin{theorem}[Edmonds {\cite[p.~123]{Edm65}}]
\label{edm-bieul}
Every eulerian graph has a bi-eulerian embedding.
\end{theorem}

Edmonds \cite{Edm65} specifically noted that Theorem \ref{edm-bieul} does not address the orientability of the embedding.  Exposing the significant differences between the orientable and nonorientable settings is one of the goals of the current work.

In the special case of $4$-regular graphs, a 1968 result of Kotzig, reinterpreted as Theorem \ref{kot-bieul} below, shows that any circuit decomposition can be completed to an embedding by adding an euler circuit.  This gives a bi-eulerian embedding when the original decomposition is an euler circuit.  Like Edmonds, Kotzig notes that his result does not address orientability.

\begin{theorem}[Kotzig {\cite[Theorem 3]{Kot68}}]
\label{kot-bieul}
Given a connected $4$-regular graph $G$ and a decomposition $\cC$ of the edge set of $G$ into circuits, there is an euler circuit $T$ such that $\cC \cup \{T\}$ is the collection of facial walks of an embedding of $G$.
\end{theorem}

For $4$-regular graphs, a bi-eulerian embedding is equivalent to a pair of euler circuits that are \emph{compatible} (also called \emph{orthogonal}) in the sense of Fleischner, Sabidussi and others (see \cite{FSW92} and \cite[pp. III.40--41]{Fle90}).  
In 1996 Andersen, Bouchet and Jackson \cite{ABJ96} provided conditions under which a $4$-regular graph embedded in the sphere, projective plane, or torus has a new bi-eulerian embedding with special properties.  They also considered $4$-regular digraphs (digraphs where every vertex has indegree $2$ and outdegree $2$) and proved some necessary conditions for the existence of a bi-eulerian directed (and hence orientable) embedding.  In 2003 Geelen, Iwata, and Murota \cite[Subsection 4.3]{GIM03} answered a question of \cite{ABJ96} by giving a polynomial-time algorithm which (as a special case, applying a result of Bouchet \cite{Bou87r}) determines whether a $4$-regular digraph has a bi-eulerian embedding.  For other related results in the $4$-regular case see \cite{BJ00, GI05, YJ22}.

In 2002 Bonnington, Conder, Morton, and McKenna proved a number of results on orientable directed embeddings of general eulerian digraphs, and the following result for tournaments.

\begin{theorem}[Bonnington et al.~{\cite[Section 3]{BCMM02}}]
\label{tournament} 
Let $T$ be a directed euler circuit in a regular tournament $D$.  Then there is an orientable directed embedding of $D$ with $T$ as one face and at most $2$ other faces.  When $|V(D)| \equiv 3 \mmod 4$, the embedding is bi-eulerian.
\end{theorem}

There is also a series of papers \cite{EGS20,GMS20,GGS05,GPS18}, beginning in 2005, by Griggs and \siran. with various coauthors (Erskine, Grannell, McCourt, and Psomas).  These involve structures from design theory (Steiner triple systems, latin squares, and symmetric configurations of block size $3$) that can be represented as triangular decompositions of graphs, and they investigate whether such a decomposition can be completed to an orientable embedding by adding an euler circuit.  They consider both undirected and directed situations.

The third reason for considering embeddings with euler circuit faces is that they are related to a number of other combinatorial objects, including other special types of graph embedding.
Bi-eulerian embeddings of graphs with $m$ edges are surface duals of embeddings of the dipole $D_m$, which consists of two vertices joined by $m$ parallel edges.  These dipole embeddings may be described by equivalence classes of permutations under cyclic shifts.  Bi-eulerian embeddings of $4$-regular graphs can be viewed as embedded medial graphs of $1$-face embeddings of $1$-vertex graphs (bouquets).  Embedded bouquets arise, for example, as terminal forms of topological extensions of the Tutte polynomial (see for example \cite{topotutte}), and in the study of twisted duals (see for example \cite{AE-M22,E-MM13,YJ22}). 

Orientable directed embeddings correspond to properly $3$-edge-colored cubic bipartite graphs, as shown by Tutte \cite{Tut48, Tut75}. This means that there is a `triality' operation defined by cyclically permuting the edge colors. Under this operation, directed embeddings with an euler circuit face become orientable directed embeddings of a $1$-vertex digraph (directed bouquet).  The correspondence with colored cubic graphs creates links between directed embeddings and a number of other structures.  For example, bi-eulerian directed embeddings of a $1$-vertex digraph correspond to properly $3$-edge-colored cubic bipartite graphs where each pair of colors induces a hamilton cycle; these then correspond to orientable quadrangular (and hence minimum genus) embeddings of the complete bipartite graph $K_{3,m}$ \cite{EW08}.

Connections as described above can be used to transfer results from one type of object to another: for example, the authors survey counting results for dipoles and bouquets (including directed bouquets) in \cite{EE-M22}, which can provide some counting results for bi-eulerian embeddings.

The main focus of the present work is embeddings in the more challenging orientable setting. 
We discuss ways in which vertices of degree $0$ mod $4$ can preclude the existence of bi-eulerian embeddings (Section \ref{sec:nec}).  On the other hand, we can always find a bi-eulerian embedding if all vertices have degree $2$ mod $4$, and some positive results can be obtained even in the presence of vertices of degree $0$ mod $4$ (Section \ref{sec:suff}).  It would be satisfying to prove more generally that a circuit decomposition can be completed to an orientable embedding using an euler circuit face if all degrees are $2$ mod $4$, but unfortunately this is false (Section \ref{sec:circdec}).

To complete the picture, we also address the easier nonorientable setting.  We show (Section \ref{sec:no-eulf}) that Theorem \ref{kot-bieul} extends to all eulerian graphs and that in most cases the embedding can be chosen to be nonorientable.  Although maximum genus nonorientable directed embeddings cannot have euler circuit faces, the same techniques used for nonorientable embeddings of undirected graphs can be applied to directed embeddings, giving straightforward results on nonorientable genus interpolation and maximum nonorientable genus (Section \ref{sec:no-diremb}).

Throughout we use whichever of graphs or digraphs give the stronger results.  Thus, when possible, we state positive results for digraphs (which have positive corollaries for undirected graphs) and negative results for undirected graphs (which have negative corollaries for digraphs).

\section{Terminology} \label{sec:term}

We begin by developing some formalism to support this work, as well as related work including \cite{EE-Mdense}.

\subsection{Representation of graphs and digraphs}

 Since we often need to specify a particular end of an edge or arc, we define graphs and digraphs using sets of vertices and half-edges or half-arcs in a framework similar to that of Fleischner~\cite{Fle90}.
 
A graph is a quadruple $G=(V,E^*,\incv,\mate)$ where $V$ (vertices) and $E^*$ (half-edges) are disjoint sets, $\incv: E^* \to V$ describes the incidence of each half-edge with a vertex, and $\mate: E^* \to E^*$ is a fixed-point-free involution that maps each half-edge to another half-edge.
An edge is an unordered pair $\{h, \mate(h)\}$ where $h \in E^*$, and we let $E$ denote the set of edges.  Note that graphs and digraphs defined in this way may have multiple edges and loops. We say an edge $e = \{h_1, h_2\}$ is \emph{incident} with the pair of (possibly equal) vertices $\incv(h_1)$ and $\incv(h_2)$, and that the vertices $\incv(h_1)$ and $\incv(h_2)$, if distinct, are \emph{adjacent}.
We use $E^*(v)$ to represent $\incv\iv(v)$, the set of half-edges incident with a given vertex $v$.  
To specify a particular graph $G$ we write $V(G)$, $E^*(G)$, $\incv_G$, $\mate_G$, $E(G)$, and $E^*_G(v)$.

Similarly, a digraph is a quintuple $D=(V, A^+, A^-, \incv, \mate)$ where $V$ (vertices), $A^+$ (outgoing half-arcs) and $A^-$ (incoming half-arcs) are disjoint sets, and if $A^* = A^+ \cup A^-$ then
$\incv : A^* \to V$ describes incidences of half-arcs with vertices, and $\mate: A^* \to A^*$ is an involution that maps each element of $A^+$ to an element of $A^-$ and vice versa.
An arc of $D$ is an ordered pair $(g, \mate(g))$ where $g \in A^+$, and we let $A$ denote the set of arcs.
We say an arc $(g,h)$ has \emph{tail} $\incv(g)$ and \emph{head} $\incv(h)$.
We use $A^+(v)$, $A^-(v)$, and $A^*(v)$ to represent the set of elements of $A^+$, $A^-$, and $A^*$, respectively, incident with a vertex $v$. To specify a particular digraph $D$ we write $V(D)$, $A^+(D)$, $A^-(D)$, $\incv_D$, $A^*(D)$, $\mate_D$, $A(D)$, $A^+_D(v)$, $A^-_D(v)$, and $A^*_D(v)$.

Since we frequently move backwards and forwards between digraphs and their underlying graphs, it is helpful to use consistent terminology.  By the `degree' of a vertex in a digraph we mean its total degree, and we refer to digraphs with both indegree and outdegree equal to $r$ at each vertex as `$2r$-regular digraphs,' rather than `$r$-regular digraphs.'  In particular, our $4$-regular digraphs have indegree and outdegree $2$ at each vertex.

It is usually easy to translate between our framework above and standard definitions for graphs and digraphs, and we warn the reader that we will use our framework or standard definitions as convenient.
Assuming such a translation, we follow the terminology of \cite{West} unless otherwise noted.
In particular, a graph, walk, digraph, or directed walk is \emph{nontrivial} if it has at least one edge or arc.

In our framework, a \emph{walk} of length $\ell$ in a graph $G$ is a sequence $W = v_0 g_1 h_1 v_1 g_2 h_2 v_2 \dots\allowbreak v_{\ell-1} g_\ell h_\ell v_\ell$ where
$v_i \in V(G)$, $g_i, h_i \in E^*(G)$,
$\incv(g_i) = v_{i-1}$, $\incv(h_i) = v_i$ and $\mate(g_i) = h_i$ for all $i$ with $1 \le i \le \ell$.
A \emph{subwalk} of $W$ is a consecutive subsequence $v_s g_{s+1} h_{s+1} v_{s+1} \dots v_{t-1} g_t h_t v_t$ of $W$ where $0 \le s \le t \le \ell$.
A \emph{directed walk} in a digraph $D$ is similar to a walk, but with conditions $g_i \in A^+(D)$ and $h_i \in A^-(D)$ for all $i$, with $1 \le i \le \ell$, and we define directed subwalks in the obvious way.
Special types of walks include \emph{closed walks} ($v_0 = v_\ell$), \emph{paths} (no repeated vertices), \emph{cycles} (closed walks that use every edge of a $2$-regular subgraph exactly once), \emph{trails} (no repeated edges), and \emph{circuits} (closed trails). When it does not cause any confusion, we do not distinguish between a closed walk and an equivalence class of closed walks under cyclic shifts, and we indicate such an equivalence class with parentheses, as $(v_0 g_1 h_1 v_1 \dots v_{\ell-1} g_\ell h_\ell)$.

An \emph{eulerian graph} has single circuit that contains all the edges and vertices, or equivalently is a connected graph with all vertices of even degree. An \emph{eulerian digraph} has a single directed circuit that contains all the arcs and vertices, or equivalently is a connected digraph with all vertices having indegree equal to outdegree.

\subsection{Embeddings and directed embeddings}

All graph embeddings in this paper are cellular. We assume that the reader is familiar with cellular embeddings of graphs in surfaces and their various combinatorial representations; standard references are \cite{E-MM13, GT, MT}. In particular, we assume familiarity with the representation of an embedding as a rotation system with positive and negative edge signatures.
We give an alternative representation of embeddings using facial walks after Lemma \ref{standardemb} below.
We also recall that the \emph{Euler genus} $\gamma$ of a surface $\Sigma$ is the genus if $\Sigma$ is nonorientable and twice the genus if $\Sigma$ is orientable, and that Euler's formula for a connected graph $G$ embedded in $\Sigma$ is $|V|-|E|+|F|=2-\gamma$.  An \emph{oriented} embedding is an embedding in an orientable surface with a specific global clockwise orientation.

If we embed a digraph $D$ so that every face is bounded by a directed closed walk of $D$ (respecting the directions of the arcs in $D$), we call this a \emph{directed embedding} of $D$.
The following basic properties of directed embeddings appear in \cite{BCMM02}.  That paper only considers orientable embeddings, but properties \ref{db-alt-s} and \ref{db-eul-s} also hold in the nonorientable case.
Property \ref{db-ori-s} (existence in the orientable case) is used implicitly in \cite{BCMM02} but not stated; we consider existence in the nonorientable case in Section \ref{sec:no-diremb}.

\begin{observation}[{Bonnington, Conder, Morton and McKenna \cite{BCMM02}}]\label{diremb-basic-s}\ 
\begin{enumerate}[(a), nosep]
\item\label{db-alt-s}
An embedding of a digraph is a directed embedding if and only if at each vertex the half-arcs, taken in rotational order, alternate between entering and leaving the vertex.

\item\label{db-eul-s}
Therefore, a connected digraph with a directed embedding must be eulerian, because at each vertex the numbers of entering and leaving half-arcs are equal.

\item\label{db-ori-s}
An eulerian digraph always has at least one orientable directed embedding, namely a rotation system with all positive edge signatures and in which the half-arcs alternate in direction around each vertex.
\end{enumerate}
\end{observation}

We can also characterize orientability of directed embeddings, as follows.

\begin{lemma}\label{diremb-ori-2fc}
A directed embedding is orientable if and only if it is $2$-face-colorable.
\end{lemma}

\begin{proof}
Suppose we have a directed embedding.  If it is orientable, orient the surface.  Then the facial directed walks may be divided into clockwise and anticlockwise walks, and since each edge must belong to one of each type, this gives a $2$-face-coloring.  Conversely, if the embedding is $2$-face-colorable, then by reversing the facial directed walks for one color class of faces, we obtain an orientation of the facial walks using every edge once in each direction, so the embedding is orientable.
\end{proof}

The `only if' part of Lemma \ref{diremb-ori-2fc} appears implicitly in \cite[p.~3]{BCMM02}, and a similar argument was used by Rongxia Hao \cite[Lemma 4.1]{H18} to prove a related result

We call a face of an oriented directed embedding a \emph{proface} or \emph{antiface} according to whether its facial directed walk is clockwise or anticlockwise, respectively. 

Our main focus is on graph embeddings with two euler circuit faces.

\begin{definition} \label{bi-eul} 
An embedding of a graph is \emph{bi-eulerian} if it has two faces, each bounded by an euler circuit.  A bi-eulerian embedding of an undirected graph may be orientable or nonorientable.
An embedding of a digraph is \emph{bi-eulerian} if it has two faces, each bounded by an euler directed circuit. A bi-eulerian embedding of a digraph is necessarily a directed embedding, and thus, by Lemma \ref{diremb-ori-2fc}, orientable.
\end{definition}

The term \emph{bi-eulerian} was used with a different meaning by Xuong \cite[p.~218]{Xuo79b}. We trust that our reuse of this term will cause no confusion.  (The term \emph{eulerian embedding} has also been used in the literature, to mean embeddings where there is one `straight-ahead' walk that is an euler circuit \cite{PTZ04}.)

\subsection{Transitions and transition graphs}\label{ssec:transition}

Some of the results below require specifying how a walk transitions from one edge to another as it passes through a vertex.  The transitions and transition graphs defined here provide the necessary formalism. 
 
\begin{definition}
Let $v$ be a vertex of a graph $G$. A \emph{transition at $v$} is an unordered pair of (possibly equal) elements of $E_G^*(v)$.
A \emph{transition graph at $v$} is a graph (multiple edges and loops allowed) with vertex set identified with the elements of $E_G^*(v)$; thus, each edge specifies a transition at $v$.
We can analogously define transitions and transition graphs at $v$ in a digraph $D$ by replacing $E_G^*(v)$ by $A_D^*(v)$ (although we will not need these in this paper).
\end{definition}

Note that these transition graphs are related to, but not the same as, the transition graphs defined in \cite{ESZtripnon}. A transition graph at $v$ that is $1$-regular, i.e., that corresponds to a partition of the half-edges into pairs, is equivalent to a \emph{transition system at $v$} as defined by Fleischner \cite[p. III.40]{Fle90}.

We now form transition graphs from walks.

\begin{definition} Let $W = (v_0 g_1 h_1 v_1 g_2 h_2 v_2 \dots v_{\ell-1} g_0 h_0)$ be a closed walk in a graph $G$.  Let $v \in V(G)$.  We form the \emph{transition graph of $W$ at $v$}, denoted $\itr(W,v)$ as follows.
If $W$ does not use $v$ then $\itr(W,v)$ is an edgeless graph with vertex set $E_G^*(v)$.
Otherwise, suppose $v$ occurs $r$ times along $W$, as vertices $v_{i(0)}, v_{i(1)}, \ldots, v_{i(r-1)}$ where $0 \le i(0) < i(1) < i(2) \dots < i(r-1) \le \ell-1$.
Then $\itr(W,v)$ contains $r$ edges, one edge joining $h_{i(j)}$ and $g_{i(j)+1}$ for each $j \in \mZ_r$.  See Figure \ref{coherent} at top right for an example, where $E_G^*(v) = \set{g_1, g_2, \dots, g_6}$.
Transition graphs are defined analogously for walks (directed or not) in digraphs.
\end{definition}

\begin{figure}[h!]
  \centering
   \hbox to \hsize{\hfil
    \includegraphics[scale=1.2]{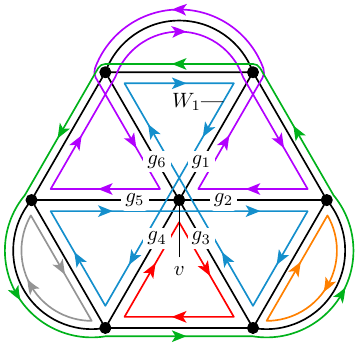}\hfil
    \vbox{\hbox{\includegraphics[scale=1.2]{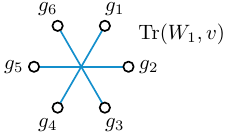}}
        \vskip20pt
        \hbox{\includegraphics[scale=1.2]{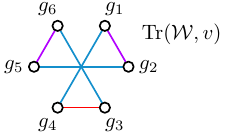}}
        \vskip10pt \hrule height0pt
    }\hfil
   }
 \caption{A collection of walks $\cW$ with transition graphs at $v$.} 
  \label{coherent}
  \end{figure} 

We now extend transition graphs to multiple walks.  

\begin{definition} Let $\cW = \{W_1, W_2, \dots, W_k\}$ be a collection (multiset) of closed walks in a graph $G$, for example, a set of facial walks.  For such a collection $\cW$, we define the \emph{transition graph of $\cW$ at $v$,} denoted $\itr(\cW,v)$, as follows.  Its vertex set is the same as that of each  $\itr(W_i,v)$, namely $E_G^*(v)$.   Its edges are the edges of all of the $\itr(W_i,v)$ for $1 \le i \le k$, with repetition.  In Figure \ref{coherent}, $\cW$ is the collection of all walks shown and $\itr(\cW, v)$ is given at bottom right.
This definition also extends to digraphs in the natural way.  
\end{definition}

We emphasize that in a general collection of walks, a given transition can occur more than once (even in a single closed walk) and the transition graph has a separate edge for each time the transition occurs, so it may have multiple edges.  Furthermore loops can arise in transition graphs corresponding to walks that enter and then immediately leave a vertex along the same edge.

\begin{definition}
A collection $\cW$ is \emph{cyclically compatible} at a vertex $v$ if $\itr(\cW,v)$ is a single cycle using every element of $E^*(v)$.  We say that $\cW$ is cyclically compatible if it is cyclically compatible at every vertex.  For example, in Figure \ref{coherent} the collection $\cW$ is cyclically compatible at $v$, and the reader may check that this also holds at the other vertices.
\end{definition}

For pairs of euler circuits, or pairs of circuit decompositions, cyclic compatibility is stronger than compatibility in the sense of Fleischner \cite[pp. III.40--41]{Fle90}. 

The following lemma states standard results on graph embeddings using our terminology.  As an illustration, the collection of walks in Figure \ref{coherent} satisfies both parts of this lemma, and therefore represents an orientable embedding of the graph (in the double torus).

\begin{lemma}\label{standardemb}
Let $\cW$ be a collection of closed walks in a graph $G$.
\begin{enumerate}[(a), nosep]
\item \label{standardemb1}
The collection $\cW$ is the collection of facial walks of some (not necessarily orientable) embedding of $G$ if and only if $\cW$ is cyclically compatible.

\item \label{standardemb2}
If $\cW$ is the collection of facial walks of an embedding of $G$, then the embedding is orientable if and only if the elements of $\cW$ can be oriented so that each edge is used once in each direction.
\end{enumerate}
\end{lemma}

Lemma \ref{standardemb} provides conditions under which a collection of facial walks describes an embedding.  This can be viewed as an alternative definition of an embedding.  In the rest of this paper we often consider embeddings from this perspective, and identify a face with its boundary walk.

\subsection{Relative embeddings and upper relative embeddings}

In some situations we will be able to construct bi-eulerian embeddings of an eulerian graph $G$ where one of the euler circuit faces is specified in advance.  It is also natural in certain settings (see Section \ref{sec:circdec}) to specify a more general collection $\cW$ of closed walks that covers every edge exactly once, and ask if $\cW$ can be completed with an euler circuit to become the collection of face boundaries of an embedding.  The collection $\cW$ must be a \emph{circuit decomposition}, partitioning $E(G)$ into circuits.  Similarly, we can try to complete a directed circuit decomposition of an eulerian digraph to a directed embedding using a directed euler circuit.  Some relevant standard terminology follows.

Suppose $G$ is a connected graph and $\cW$ is a collection of closed walks in $G$.  If $\cW'$ is another collection of closed walks in $G$ such that $\cW$ and $\cW'$ together form the facial walks of an embedding $\Phi$, then we say $\Phi$ is an \emph{embedding of $G$ relative to $\cW$}, and that the faces bounded by elements of $\cW$ are \emph{inner} faces, while those bounded by elements of $\cW'$ are \emph{outer} faces.
If $\Phi$ is orientable, $\cW$ is a circuit decomposition of $G$, and there are only one or two outer faces, then we say that $\Phi$ is an \emph{upper embedding of $G$ relative to $\cW$}.  (We require $\cW$ to be a circuit decomposition to avoid complications that arise in more general situations.)  In particular, if we complete $\cW$ to an orientable embedding by adding an euler circuit, then we have an upper relative embedding.
If an upper embedding of $G$ relative to $\cW$ exists, it has maximum genus over all orientable embeddings of $G$ relative to $\cW$.  We can also apply these concepts to collections of directed walks, directed circuit decompositions, and directed embeddings of digraphs.

\section{Necessary conditions for orientable bi-eulerian embeddings}
\label{sec:nec}

In Sections \ref{sec:nec} through \ref{sec:circdec} we examine the existence of bi-eulerian and related embeddings in orientable surfaces.  The results largely concern the degrees of the vertices in eulerian graphs or digraphs, and in particular how many vertices have degrees congruent to $0$ mod $4$ as opposed to $2$ mod $4$.

We begin with an observation based on Euler's formula, which leads to some simple necessary conditions for an orientable bi-eulerian embedding.  These conditions generalize \cite[Lemma 8 and Corollary 9]{ABJ96} and \cite[Corollary 3.10]{YJ22}.

\begin{observation}\label{faceszeroverts}
Suppose we have an $n$-vertex $m$-edge eulerian graph or digraph with $\ell$ vertices of degree $0$ mod $4$.  By Euler's formula, an embedding with $r$ faces in an orientable surface of genus $g$ must have $r = m - n + 2 - 2g = \sum_{v \in V(G)} (\hf \deg(v)-1) + 2-2g$.  Vertices of degree $0$ mod $4$ make an odd contribution to the sum, while those of degree $2$ mod $4$ make an even contribution.  Thus, $r \equiv m-n \equiv \ell$ (mod $2$).
\end{observation}

Since a bi-eulerian embedding has two faces, we have the following immediate consequence. 

\begin{proposition}\label{evenzeroverts}
If an $n$-vertex $m$-edge graph or a digraph  has an orientable bi-eulerian embedding, then the number of vertices of degree $0 \mmod 4$ is even (or equivalently $m-n$ is even).
\end{proposition}

Recall that an \emph{edge cut} in a connected graph $G$ is a set of edges $C$ each having one endpoint in  a proper nonempty subset $U$ of $V(G)$ and the other in the complement $V(G)-U$. We say $C$ is a \emph{$k$-edge cut} if $|C|=k$.

Suppose $G$ has a $2$-edge cut $\set{e,f}$ where $e = \set{e_1, e_2}$, $f = \set{f_1, f_2}$, and where $e_1$ and $f_1$ are the half-edges incident with vertices of $U$, and $e_2$ and $f_2$ are the half-edges incident with vertices of $V(G)-U$.  Then the \emph{$2$-edge-cut reduction} of $\set{e,f}$ replaces $e$ and $f$ by new edges $g_1=\set{e_1, f_1}$ and $g_2=\set{e_2, f_2}$, giving a graph with two components: $G_1$ on vertex set $U$ and $G_2$ on vertex set $V(G)-U$.  We also say that $G$ is a \emph{$2$-edge-join} of $G_1$ and $G_2$.  If $G$ is eulerian then both $G_1$ and $G_2$ are eulerian.

If $D$ is a connected digraph, a $k$-edge cut $C$ in $D$ again means a set of $k$ arcs joining $U$ to $V(D)-U$ where $U$ is a proper nonempty subset of $V(D)$.
If $D$ is eulerian then $k$ must be even, and $C$ must have $k/2$ arcs from $U$ to $V(D)-U$ and $k/2$ arcs from $V(D)-U$ to $U$.
Thus, if we have a $2$-edge cut in an eulerian digraph $D$, we can still perform a $2$-edge cut reduction to obtain two eulerian digraphs $D_1$ and $D_2$.  Conversely, we can also take a $2$-edge-join of two eulerian digraphs $D_1$ and $D_2$ to obtain a new eulerian digraph $D$.

\begin{observation}\label{2ecreduction}
Suppose $G$ is an eulerian graph and we perform a $2$-edge-cut reduction on $G$ to obtain $G_1$ and $G_2$.  Then $G$ has an orientable bi-eulerian embedding if and only if both $G_1$ and $G_2$ have orientable bi-eulerian embeddings.  This follows by cutting and splicing face boundaries.  Similarly, suppose $D$ is an eulerian digraph and we perform a $2$-edge-cut reduction on $D$ to obtain $D_1$ and $D_2$.  Then $D$ has a (necessarily orientable) bi-eulerian directed embedding if and only if both $D_1$ and $D_2$ have (necessarily orientable) directed bi-eulerian embeddings.
\end{observation}

Applying Observation \ref{2ecreduction} and Proposition \ref{evenzeroverts} gives the following.

\begin{proposition}\label{badcut}
A graph or digraph with an orientable bi-eulerian embedding has no 2-edge cut with an odd number of vertices of degree congruent to $0 \mmod 4$ in each component of the 2-edge cut reduction. 
\end{proposition}
 
In other words, a $2$-edge cut with an odd number of vertices of degree $0 \mmod 4$ on one or both sides of the cut, which we call a \emph{bad $2$-edge cut}, prevents the existence of an orientable bi-eulerian embedding.

An graph or digraph is \emph{admissible} if it is eulerian and satisfies the necessary conditions given by Propositions~\ref{evenzeroverts} and~\ref{badcut}, i.e., it has an even number of vertices of degree $0$ mod $4$ and no bad $2$-edge-cut.  Admissibility is not sufficient to guarantee the existence of an orientable bi-eulerian embedding, as the following two examples show.

\begin{example}\label{cycledigons}
For $n\ge 2$ let $\ddc_n$ be the $4$-regular eulerian digraph obtained by replacing each edge in an $n$-cycle by a directed digon (two arcs in opposite directions).  Figure \ref{CounterEx}(a) depicts $\ddc_4$.
If $n$ is even then $\ddc_n$ has an even number of vertices of degree $0$ mod $4$ and no $2$-edge-cut (bad or otherwise), so it is admissible.  But the following analysis shows that $\ddc_n$ has an orientable bi-eulerian directed embedding only if $n=2$.  We analyze the number of faces for more general circuit decompositions as the general technique will also apply in later settings. 

Label the vertices of $\ddc_n$ as $v_0, v_1, v_2, \dots, v_{n-1}$ (subscripts interpreted modulo $n$) and suppose we have arcs $a_i$ from $v_i$ to $v_{i+1}$ and $b_i$ from $v_{i+1}$ to $v_i$ for each $i \in \set{0, 1, \dots, n-1}$.
By Lemmas \ref{diremb-ori-2fc} and \ref{standardemb}, an orientable directed embedding of $\ddc_n$ corresponds to two directed circuit decompositions $\cC_1$, $\cC_2$ such that $\cC_1 \cup \cC_2$ is cyclically compatible.  Say $\cC_i$ \emph{splits} at $v_j$ if $\cC_i$ has a circuit with consecutive arcs $a_{j-1}, b_{j-1}$ and a circuit (possibly the same circuit) with consecutive arcs $b_j, a_j$.
Exactly one of $\cC_1$ or $\cC_2$ must split at each vertex.  Assume each $\cC_i$ splits at $n_i$ vertices.  If $n_i = 0$ then $\cC_i$ consists of two hamilton cycles.  If $n_i \ge 1$ then $\cC_i$ consists of exactly $n_i$ circuits.  For a bi-eulerian embedding, we must have $n_1 = n_2 = 1$ and so $n = n_1+n_2 = 2$.
\end{example}

We can also construct undirected graphs with $6$ or more vertices of degree $0$ mod $4$ that satisfy the conditions from Propositions \ref{evenzeroverts} and \ref{badcut} but have no orientable bi-eulerian embedding, as in the following example.

\begin{example}\label{exforbconf}
Figure \ref{CounterEx}(b) shows a graph that is admissible but does not have an orientable bi-eulerian embedding.  This example was found by a computer search, and is the unique $4$-edge-connected eulerian graph on $12$ or fewer edges with this property.
\end{example}

\begin{figure}[h!]
\centering
\hbox to\hsize{(a)\hfil (b)\hfil \kern15pt(c)\kern-15pt\hfil}
\vskip-1.3\baselineskip
\hbox to\hsize{
    \qquad\raise15pt\hbox{\includegraphics[scale=1.2]{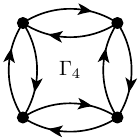}}\hfil
    \includegraphics[scale=1.2]{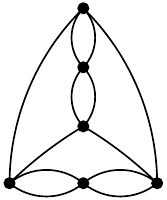}\hfil
    \raise7pt\hbox{\includegraphics[scale=1.2]{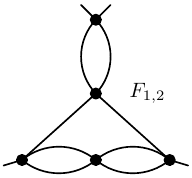}}
}
\caption{(a) $\ddc_4$. (b) Graph with no bi-eulerian embedding. (c) $F_{1,2}$. } 
\label{CounterEx}
\end{figure}

In fact,  there is an infinite family of configurations whose presence prevents the existence of an orientable bi-eulerian embedding even if the graph is admissible.
A \emph{chain of digons of length $s$} in a graph $G$ consists of $s+1$ vertices $v_0, v_1, \dots, v_s$, all of which have degree $4$ in $G$, such that $v_{i-1}$ and $v_i$ are joined by a digon (pair of parallel edges) for $1 \le i \le s$.  The configuration $F_{s,t}$ consists of a chain of digons of length $s$ and a chain of digons of length $t$, where the last vertex of the first chain is adjacent to both the first and last vertices of the second chain.
For example, $F_{1,2}$ is shown in Figure \ref{CounterEx}(c) (the first chain of digons is vertical, the second is horizontal).  It is the smallest configuration in this family whose presence prevents the existence of an orientable bi-eulerian embedding.
The graph of Figure \ref{CounterEx}(b) contains $F_{1,2}$.

\begin{proposition}\label{forbconf}
Suppose $G$ is an eulerian graph containing a configuration $F_{s,t}$ where either $s=1$ and $t$ is even, or $s \ge 2$ and $t \ge 2$.  Then $G$ has no orientable bi-eulerian embedding.
\end{proposition}

The proof, given in detail in the Appendix, 
%\ref{app:forbconf}, 
involves a case analysis of how the euler circuits forming the face boundaries of a potential orientable bi-eulerian embedding would pass through the two chains of digons.  The ideas are similar to those used in Example \ref{cycledigons}, but the analysis is more complicated because edges do not have specified directions.

\section{Sufficient conditions for orientable bi-eulerian embeddings}
\label{sec:suff}

To find maximum genus directed embeddings it is natural to leverage successful techniques for finding maximum orientable genus embeddings for undirected eulerian graphs.  A vertex identification method due to \yavorskii. \cite{Yav73} was used by  Glukhov \cite{Glu77} to show that connected graphs with all vertices of degree $2 \mmod 4$ have a $2$-face orientable embedding, and so are upper embeddable.
\skoviera. and Nedela \cite{SN90} extended Glukhov's result to all eulerian graphs, showing how vertices of degree $0$ mod $4$ can affect the minimum number of faces in an embedding, and hence the maximum genus. We adapt the vertex identification technique to directed embeddings, in a controlled way that allows us to produce bi-eulerian embeddings.  As we might expect from Section \ref{sec:nec}, the vertex degrees modulo $4$ play an important role.

\subsection{Vertex identifications for directed embeddings}

Given two digraphs $D$ and $F$, we say that a function $f: V(D) \to V(F)$ is a \emph{vertex identification from $D$ to $F$} if $f$ is a surjection, and $F$ is obtained from $D$ by identifying $f\iv(w) \subseteq V(D)$ into a single vertex $w$ for all $w \in V(F)$.
We write $F = D/f$.  Note that $D$ and $F$ have the same half-arcs and arcs, and their incidence functions satisfy $\psi_F = f \circ \psi_D$.
If $W$ is a walk in $D$, then $W/f$ denotes the corresponding walk in $D/f$, where we keep the same half-arcs but replace each vertex $v$ by $f(v)$.

Lemmas \ref{ident2vert} and \ref {ident3vert} and Corollary \ref {identvert} that follow describe the effect of some vertex identifications on orientable directed embeddings. The proofs use concatenations of walks that share a vertex, so that if $A = u g \dots h v   $ and $B = u'g'\dots h'v' $ are walks with $v=u'$, then $A \mw B$ is the walk $ u g \dots h v g'\dots h'v'$.  

\begin{lemma}[Identifying two vertices]
\label{ident2vert}
Let $\Phi$ be an orientable directed embedding of an eulerian digraph $D$, and suppose $f$ identifies distinct vertices $v_1, v_2$ of $D$ into a single vertex $v$, while leaving other vertices unchanged.  Let $A_1$ and $A_2$ be antifaces of $D$ incident with $v_1$ and $v_2$, respectively, and let $\cA$  be the set containing $A_1$ and $A_2$ (which may or may not be distinct).
Then $D/f$ has an orientable directed embedding $\Phi'$ with faces as follows.
\begin{enumerate}[(a), nosep]
\item If  $A_1 \ne A_2$, so $|\cA| = 2$, then $\Phi'$ has a single antiface replacing $A_1$ and $A_2$.

\item If $A_1 = A_2$, so $|\cA| = 1$,  then $\Phi'$ has two antifaces replacing the single antiface $A_1=A_2$.

\end{enumerate}
In both cases the set of new antifaces $\cA'$ uses the same set of half-arcs as $\cA$.  Furthermore, $\Phi'$ has a face $B/f$ for every face $B \notin \cA$ of $\Phi$, and $|\cA'| \le |\cA|+1$.
\end{lemma}

\begin{figure}[ht!]
  \centering
  \begin{subfigure}{\textwidth}
    \hbox to\hsize{\hfil\includegraphics[scale=1.2]{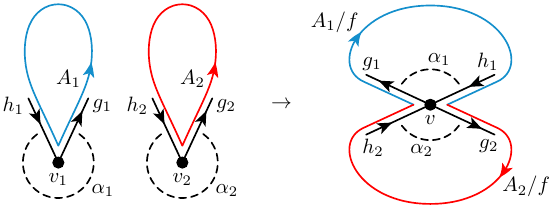}\hfil}
    \subcaption{$A_1 \ne A_2$, i.e., $|\cA|=2$.}\label{i2a}
  \end{subfigure}
  \begin{subfigure}{\textwidth}
    \hbox to \hsize{\hfil\includegraphics[scale=1.2]{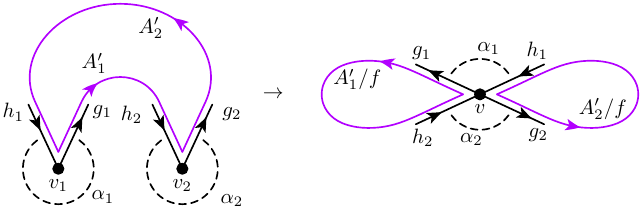}\hfil}
    \subcaption{$A_1 = A_2$, i.e., $|\cA|=1$.}\label{i2b}
  \end{subfigure}
  \caption{Cases for Lemma \ref{ident2vert}.}
\label{fig-identify2}
\end{figure}

\begin{proof}
Since each $A_i$ is incident with $v_i$, we may assume that it has the form $(\dots h_i v_i g_i \dots)$, or equivalently $v_i g_i \dots h_i v_i$, for some incoming half-arc $h_i$ and outgoing half-arc $g_i$.  Since $A_i$ is an antiface, the face bounded by $A_i$ must appear immediately after $h_i$ and immediately before $g_i$ in the rotation (clockwise sequence of half-arcs) determined by $\Phi$ at $v_i$.  The full rotation at $v_i$ can be written as $(g_i \al_i h_i)$, where $\al_i$ is the (possibly empty) sequence of the rest of the half-arcs incident with $v_i$.  

We define $\Phi'$ by letting the rotation at $v$ be $(g_1 \al_1 h_1 g_2 \al_2 h_2)$, and leaving the rotations of all the other vertices unchanged.  We now observe the effect of this new rotation scheme on the facial walks.  See Figure \ref{fig-identify2}.

(a) If $A_1 \ne A_2$ then think of each $A_i$ as a walk $v_i g_i \dots h_i v_i$. The new rotation at $v$ means that in $\Phi'$, $A_1$ and $A_2$ are replaced by one antiface with boundary $(A_1/f) \mw (A_2/f)$.  See Figure \ref{i2a}.

(b) If $A_1 = A_2 = A$, then, since $v_1$ and $v_2$ are distinct, we may view $A$ as consisting of two trails, one from $v_1$ to $v_2$ followed by one from $v_2$ to $v_1$.  Thus, we can write $A = A_1' \mw A_2'$ where $A_1' = v_1 g_1 \dots h_2 v_2$ and $A_2' = v_2 g_2 \dots h_1 v_1$.  Then $A_1'/f$ and $A_2'/f$ are antifaces in $\Phi'$ replacing $A$.  See Figure \ref{i2b}.
\end{proof}

\begin{lemma}[Identifying three vertices]
\label{ident3vert}
Let $\Phi$ be an orientable directed embedding of an eulerian digraph $D$, and suppose $f$ identifies distinct vertices $v_1, v_2, v_3$ of $D$ into a single vertex $v$, while leaving other vertices unchanged.  Let $A_i$ be an antiface incident with $v_i$ for $i \in \set{1, 2, 3}$, and let $\cA$ be the set containing  $A_1$, $A_2$, and $A_3$, which may not be distinct.
Then $D/f$ has an orientable directed embedding $\Phi'$ with faces as follows.
\begin{enumerate}[(a), nosep]

\item If $|\cA| = 3$ then $\Phi'$ has a single antiface replacing the antifaces in $\cA$.

\item If $|\cA| = 2$ then $\Phi'$ has two new antifaces replacing the antifaces in $\cA$.

\item If $|\cA| = 1$ then $\Phi'$ has a single new antiface replacing the antiface in $\cA$.

\end{enumerate} In all cases the set of new antifaces $\cA'$ uses the same set of half-arcs as $\cA$.  Furthermore, $\Phi'$ has a face $B/f$ for every face $B \notin \cA$ of $\Phi$, and $|\cA'| \le |\cA|$.
\end{lemma}

\begin{figure}[ht!]
  \centering
  \begin{subfigure}{\textwidth}
    \hbox to\hsize{\hfil\includegraphics[scale=1.2]{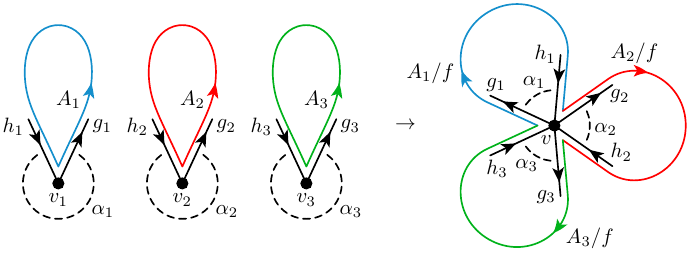}\hfil}
    \subcaption{$|\cA|=3$.}\label{i3a}
  \end{subfigure}
  \begin{subfigure}{\textwidth}
    \hbox to \hsize{\hfil\includegraphics[scale=1.2]{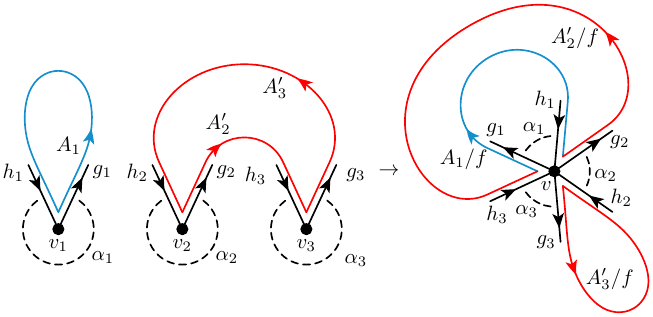}\hfil}
    \subcaption{$|\cA|=2$.}\label{i3b}
  \end{subfigure}
    \begin{subfigure}{\textwidth}
    \hbox to \hsize{\hfil\includegraphics[scale=1.2]{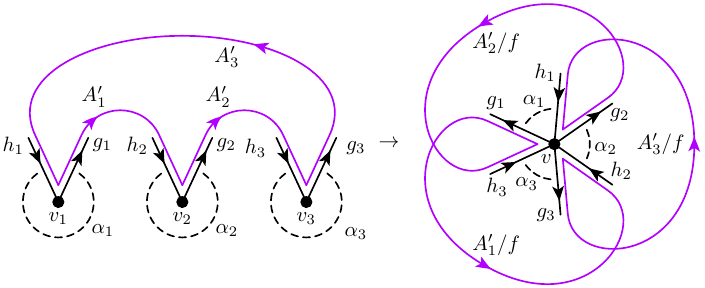}\hfil}
    \subcaption{$|\cA|=1$.}\label{i3c}
  \end{subfigure}
  \caption{Cases for Lemma \ref{ident3vert}.}
\label{fig-identify3}
\end{figure}

\begin{proof}
Since each $A_i$ is incident with $v_i$, as in the proof of Lemma \ref{ident2vert}, we may assume that it has the form $v_i g_i \dots h_i v_i$, for some incoming half-arc $h_i$ and outgoing half-arc $g_i$.  The rotation (clockwise sequence of half-arcs) determined by $\Phi$ at each $v_i$ can then be written as $(g_i \al_i h_i)$, where $\al_i$ is the (possibly empty) sequence of the rest of the half-arcs incident with $v_i$. We now observe the effect of this new rotation scheme on the facial walks.  See Figure \ref{fig-identify3}.   

(a) If $|\cA| = 3$, consider each $A_i$ as a walk $v_i g_i \dots h_i v_i$.  If we define the rotation at $v$ as $(g_1\al_1h_1g_2\al_2h_2g_3\al_3h_3)$ then $A_1, A_2, A_3$ are replaced by a single antiface with boundary $(A_1/f)\mw (A_2/f) \mw (A_3/f)$.  See Figure \ref{i3a}.

(b) If $|\cA| = 2$, suppose without loss of generality that $A_2 = A_3$.  Consider $A_1$ as the walk $v_1 g_1 \dots h_1 v_1$, and write $A_2$ as $A_2' \mw A_3'$ where $A_2' = v_2 g_2 \dots h_3 v_3$ and $A_3' = v_3 g_3 \dots h_2 v_2$.  If we define the rotation at $v$ as $(g_1 \al_1 h_1 g_2 \al_2 h_2 g_3 \al_3 h_3)$ then $A_1$ and $A_2$ are replaced by two antifaces $(A_1/f) \mw (A_2'/f)$ and $A_3'/f$.  See Figure \ref{i3b}.

(c) If $|\cA| = 1$, we may suppose without loss of generality that $A_1 = A_1' \mw A_2' \mw A_3'$ where $A_1' = v_1 g_1 \dots h_2 v_2$, $A_2' = v_2 g_2 \dots h_3 v_3$, and $A_3' = v_3 g_3 \dots h_1 v_1$.  If we define the rotation at $v$ as $(g_1 \al_1 h_1 g_2 \al_2 h_2 g_3 \al_3 h_3)$ then $A_1$ is replaced by one antiface $(A_1'/f) \mw (A_3'/f) \mw (A_2'/f)$.  See Figure \ref{i3c}.
\end{proof}

In case (c) of Lemma \ref{ident3vert}, when $|\cA|=1$, we can also construct an embedding $\Phi'$ where the single antiface in $\cA$ is replaced by three new antifaces. However, our goal is to minimize the number of antifaces, so we will not use this.

As we show below, situations where we identify an odd number of vertices into a single vertex behave nicely.  Therefore, following \skoviera. and Nedela \cite{SN90}, we say a vertex identification $f$ from $D$ to $F$ is an \emph{odd vertex identification with $\ell$ exceptional vertices} if there are exactly $\ell$ vertices $w \in V(F)$ for which $|f\iv(w)|$ is even, where $\ell \ge 0$.

\begin{corollary}
\label{identvert}
Let $\Phi$ be an orientable directed embedding of an eulerian digraph $D$, and suppose $f$ is an odd vertex identification from $D$ to $D/f$ with $\ell$ exceptional vertices.
Let $\cA$ be a set of antifaces such that for each $w \in V(D/f)$ with $|f\iv(w)| \ge 2$, each element of $f\iv(w)$ is incident with an element of $\cA$.
Then $D/f$ has an orientable directed embedding $\Phi'$ where $\cA$ is replaced by a set $\cA'$ of antifaces, such that $|\cA'| \le |\cA|+\ell$.
The set $\cA'$ uses the same set of half-arcs as $\cA$, and $\Phi'$ has a face $B/f$ for every face $B \notin \cA$ of $\Phi$.
\end{corollary}

\begin{proof}
Consider the vertices of $D/f$ one at a time.  Suppose $w \in V(D/f)$ and $|f\iv(w)| = k$. If $k=1$ we do nothing (except relabel the unique vertex of $f\iv(w)$ as $w$).

If $k=2t+1 \ge 3$ is odd then we can identify $f\iv(w)$ into the single vertex $w$ by a sequence of $t$ $3$-vertex identifications.  Applying Lemma \ref{ident3vert} we can construct a corresponding sequence of orientable directed embeddings, where the number of antifaces derived from $\cA$ does not increase at each step.

If $k=2t+2$ is even then in addition to $t$ $3$-vertex identifications as above, we must also use one $2$-vertex identification and apply Lemma \ref{ident2vert}, which adds at most $1$ to the number of antifaces derived from $\cA$.

Processing each vertex of $D/f$ in turn, the overall number of antifaces derived from $\cA$ increases by at most $\ell$.
\end{proof}

\subsection{A sufficient condition and a sharp bound}

We now give our main results in the orientable setting.  Theorem \ref{2mod4main} shows that a sufficient condition for an eulerian digraph to have a (necessarily orientable) bi-eulerian directed embedding is that all vertices have degree $2 \mmod 4$.  Furthermore, one of the euler circuits bounding a face may be specified in advance.  Theorem \ref{2mod4main} is a special case of the more general result of Theorem \ref{maxgenus0mod4} that also addresses vertices of degree $0 \mmod 4$, giving a sharp bound on maximum orientable directed genus.

Our proof relies on Corollary \ref{identvert}, which is based on Lemmas \ref{ident2vert} and \ref{ident3vert}, which specify explicit procedures for modifying an embedding.  Thus, implicit in our proofs are polynomial time algorithms for producing the desired embeddings.

We prove our results, here and in the next subsection, in the digraph setting.  But since every eulerian graph can be oriented to give an eulerian digraph, immediate corollaries for undirected graphs follow.

\begin{theorem}
\label{maxgenus0mod4}
Let $D$ be an $m$-edge $n$-vertex eulerian digraph in which exactly $\ell$ of the vertices have degree $0$ mod $4$.  Let $T$ be a given directed euler circuit in $D$.  Then there exists an orientable directed embedding of $D$ with one proface bounded by $T$ and at most $\ell+1$ antifaces.
Therefore, the maximum orientable directed genus of $D$ is at least $(m-n-\ell)/2$.
\end{theorem}

\begin{figure}[h!]
  \centering
  \hbox to\hsize{\hfil
    \includegraphics[scale=1.1]{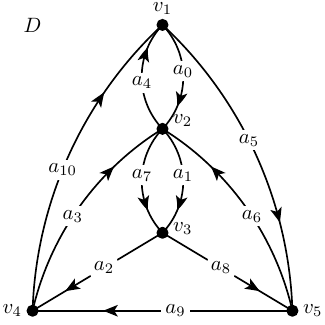}\hfil\hfil
    \raise10pt\hbox{\includegraphics[scale=1.1]{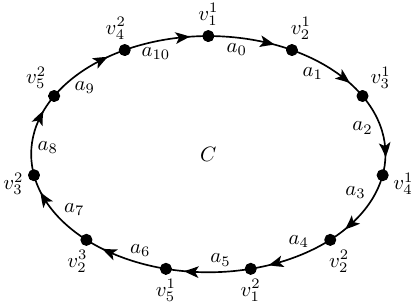}}\hfil
  }
 \caption{Digraph $D$ is the image of a vertex identification map from directed cycle $C$.} 
  \label{pullapart}
  \end{figure} 

\begin{proof}
Since $D$ is eulerian, there is a directed cycle digraph $C$ and a vertex identification map $f$ from $C$ to $D$ such that $C/f = D$ and $Z/f = T$, where $Z$ is the directed cycle walk in $C$.  The cycle $C$ has a planar bi-eulerian directed embedding $\Phi$ where $Z$ bounds the proface, and the antiface $A$ is also bounded by $Z$.  See for example Figure \ref{pullapart}, where we have a digraph $D$ with euler circuit $T$ using arcs $a_0, a_1, a_2, \dots, a_{10}$ in that order, and which can be obtained from a directed cycle digraph $C$ with a directed cycle walk $Z$ also using arcs $a_0, a_1, a_2, \dots, a_{10}$ in that order, with $f(v_i^j) = v_i \in V(D)$ for each $v_i^j \in V(C)$.

For each $v \in V(D)$, we have that $|f\iv(v)| = \deg_D(v)/2$, so $f$ is an odd vertex identification with $\ell$ exceptional vertices. 
Applying Corollary \ref{identvert} to $\Phi$ and $\cA=\{A\}$, we obtain an orientable directed embedding $\Phi'$ of $C/f = D$ with at most $\ell+1$ antifaces derived from $A$, and where $Z/f = T$ is also a face.
Thus, $\Phi'$ is the required embedding.  It has at most $\ell+2$ faces, so the bound on the maximum genus follows from Euler's formula.
\end{proof}

If all the vertices of a diagraph have degree 2 mod 4, then $\ell=0$ and $D$ has a bi-eulerian embedding, as follows.

\begin{theorem}\label{2mod4main}
Let $D$ be an eulerian digraph with all vertices of degree congruent to $2$ mod $4$, and let $T$ be any directed euler circuit of $D$.  Then $D$ has a (necessarily orientable) bi-eulerian directed embedding with one of the faces bounded by $T$.
\end{theorem}

\begin{example}
Sharpness examples for Theorem \ref{maxgenus0mod4} are provided by chains of directed digons.  Let $\DP_\ell$ have  vertices $v_0, v_1, v_2, \dots, v_{\ell+1}$,
with one arc from $v_i$ to $v_{i+1}$ and one arc from $v_{i+1}$ to $v_i$
for each $i \in \set{0, 1, 2, \ldots, \ell}$.  There are $\ell$ vertices of degree $0$ mod $4$.  There is only one directed euler circuit in $\DP_\ell$, and the only way to complete this to an orientable directed embedding is to use all $\ell+1$ directed digons in $\DP_\ell$ as faces.

In fact every orientable directed embedding of $\DP_\ell$ has exactly $\ell+2$ faces, so the maximum (and minimum) genus of orientable directed embeddings of $\DP_\ell$ is $(m-n-\ell)/2 = 0$ (all are planar).  This can be shown using an analysis similar to that of Example \ref{cycledigons}.
\end{example}

The digraphs $\DP_\ell$ have many bad $2$-edge-cuts.  However, we cannot improve the bound of Theorem \ref{maxgenus0mod4} by more than $2$ by prohibiting bad $2$-edge-cuts.
The digraphs $\ddc_\ell$ of Example \ref{cycledigons} have an underlying graph that is $4$-edge-connected, and $\ell$ vertices of degree $0$ mod $4$.  The analysis from that example shows that every orientable directed embedding of $\ddc_\ell$ with an euler circuit face has exactly $\ell-1$ other faces, and in fact every orientable directed embedding has at least $\ell$ faces.  Subsection \ref{ssec:two0mod4} below addresses the $4$-edge-connected situation when there are exactly two vertices of degree $0$ mod $4$.

We can prove Theorems \ref{maxgenus0mod4} and \ref{2mod4main}, and also Theorem \ref{2badok} below, using alternative arguments.  Our first proofs used an approach inspired by Fleischner's techniques \cite[Section VI.2]{Fle90} for constructing compatible euler circuits by specifying transition graphs built from $1$-factors (perfect matchings).
We can also give proofs using results of Bonnington \cite{Bon94} on maximum genus relative embeddings for undirected graphs.  The proofs given here, based on the idea of odd vertex identifications from \cite{Glu77, SN90, Yav73}, are the simplest ones we have found.

We can now apply the digraph results to graphs.  Given an euler circuit $T$ in an undirected eulerian graph $G$, we can orient the edges of $G$ to follow $T$, giving a digraph $D$.  Applying Theorem \ref{maxgenus0mod4} to $D$ then yields the following.

\begin{corollary}\label{mod4undir}
Let $G$ be an eulerian graph, and let $T$ be any euler circuit of $D$. Suppose $D$ has exactly $\ell$ vertices of degree $0$ mod $4$.  Then $G$ has an orientable embedding with a face bounded by $T$ and at most $\ell+1$ other faces.
When all vertices have degree $2$ mod $4$, this is a bi-eulerian embedding.
\end{corollary}

\subsection{Bi-eulerian embeddings with two vertices of degree $0$ mod $4$}
\label{ssec:two0mod4}

Even when there are vertices of degree $0$ mod $4$,  bi-eulerian embeddings may still exist.
Let $\dip_4$ be the eulerian digraph with two vertices $u, v$, and with two arcs $a_1, a_2$ from $u$ to $v$ and two arcs $b_1, b_2$ from $v$ to $u$.
Then $\dip_4$ has two vertices of degree $0$ mod $4$, but still has a bi-eulerian embedding.  In fact, the bi-eulerian embedding of $\dip_4$ can be used to find bi-eulerian embeddings in other digraphs with two vertices of degree $0$ mod $4$.

A key concept here is interlacing: a circuit \emph{interlaces} two vertices $x$ and $y$ if it can be written as $(\dots x \dots y \dots x \dots y \dots)$.

\begin{theorem} \label{2badok}
Let $D$ be an eulerian digraph with all but two vertices of degree
congruent to $2$ mod $4$, and let $T$ be a directed euler circuit of $D$ that interlaces the two vertices of degree $0$ mod $4$. Then $D$ has a (necessarily orientable) bi-eulerian directed embedding with one of the faces bounded by $T$. 
\end{theorem}

\begin{proof}
Since $D$ is eulerian, there is a directed cycle digraph $C$ and a vertex identification map $f$ from $C$ to $D$ such that $C/f = D$ and $Z/f = T$, where $Z$ is the directed cycle walk in $C$.  

Suppose $x$ and $y$ are interlaced on $T$.  Then there are $x_1, x_2 \in f\iv(x)$ and $y_1, y_2 \in f\iv(y)$ such that $x_1, y_1, x_2, y_2$ occur along $Z$ in that order.
Let $f_1$ be a vertex identification of $C$ that identifies $x_1$ and $x_2$ into a vertex $x'$, and $y_1$ and $y_2$ into a vertex $y'$.  Let $C' = C/f_1$ and $Z' = Z/f_1$.
Then $C'$ is isomorphic to a subdivision of $\dip_4$, with $Z'$ corresponding to an euler circuit $X_1$ of $\dip_4$ which we may assume is represented by the vertex/arc sequence $(u a_1 v b_1 u a_2 v b_2)$.  Now $\dip_4$ has an orientable bi-eulerian directed embedding with $X_1$ as the proface and $X_2 = (u a_1 v b_2 u a_2 v b_1)$ as the antiface.  Thus, $C'$ has a corresponding bi-eulerian embedding $\Phi$ with $Z'$ as the proface and with one antiface $A$ corresponding to $X_2$.

There is a unique vertex identification $f_2$ from $C'$ to $D$ such that $f = f_2 \circ f_1$, and $f_2$ is an odd vertex identification with no exceptional vertices.  Applying Corollary \ref{identvert} to $\Phi$, $C'$, $f_2$ and $\cA=\{A\}$, we obtain an orientable directed embedding $\Phi'$ of $C'/f_2 = D$ with one antiface derived from $A$, and where $Z'/f_2 = T$ is also a face.  This is the required embedding.
\end{proof}

Theorem \ref{2badok} is sharp in two ways.
First, Example \ref{TwoBadUnlaced} below shows that even if there are only two vertices of degree $0$ mod $4$ in an eulerian digraph, specifying an euler circuit $T$ that does not interlace them can preclude a bi-eulerian embedding with $T$ as one of the faces.  Second, the digraphs $\ddc_n$ of Example \ref{cycledigons} (for even $n \ge 4$) show that if an eulerian digraph has a number of vertices of degree $0$ mod $4$ that is even and more than two, then there may be no bi-eulerian embedding, regardless of the starting euler circuit $T$, even if for every pair of vertices there exists a directed euler circuit that interlaces them,

\begin{example}\label{TwoBadUnlaced} Figure \ref{UnlacedBadVerts} shows a digraph with two vertices $u, v$ of degree 0 mod 4, and an euler circuit $T$,  shown in red, that does not interlace them.   For any circuit(s) completing $T$ to a directed embedding, the transitions shown in blue are forced, and they prevent a circuit from going between the left and right sides of the digraph. Thus, any directed embedding with $T$ bounding one of its faces must have at least 3 faces and hence is not bi-eulerian.

\begin{figure}[h!]
 \centering
 \includegraphics[scale=1.2]{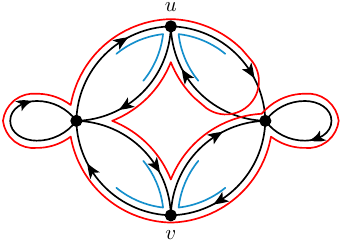}
\caption{A digraph with an euler circuit (red), that does not interlace two vertices of degree 0 mod 4, and forced transitions (blue). } 
 \label{UnlacedBadVerts}
\end{figure} 

\end{example}

The following lemmas show that provided two vertices are `locally' 4-edge connected, then it is always possible to find an euler circuit that interlaces them.  We denote the underlying undirected graph of a digraph $D$ by $\ug{D}$.
The proof of our first lemma, Lemma \ref{kpaths}, generalizes the proof of \cite[Lemma 10.7.8]{BJG09}. 

\begin{lemma} \label{kpaths}
Let $D$ be an eulerian digraph, and $s,t \in V(D)$ with $s \ne t$.  Suppose that every edge cut of $\ug{D}$ separating $s$ and $t$ has at least $2k$ edges, where $k \ge 1$.  Then $D$ has arc-disjoint directed paths $P_1, P_2, \dots, P_k, Q_1, Q_2, \dots, Q_k$, where each $P_i$ is a directed $st$-path and each $Q_i$ is a directed $ts$-path.  
\end{lemma}

\begin{proof} Since $D$ is eulerian, for every edge cut of $\ug{D}$ half of the corresponding arcs go in each direction in $D$.  Therefore the minimum size of an $st$-cut in $D$ is at least $k$, and by Menger's Theorem $D$ has $k$ arc-disjoint $st$-paths $P_1, P_2, \dots, P_k$.  Now consider a unit flow $f$ in $D - \bigcup_{i=1}^k A(P_i)$.  This flow is balanced at all vertices, except that it has net inflow of $k$ at $s$ and net outflow of $k$ at $t$.  By Gallai's Flow Decomposition Theorem (see \cite[Theorem 4.3.1]{BJG09} or \cite[Theorem 11.1]{Sch03a}), $f$ can be decomposed into unit flows around some directed cycles and along $k$ arc-disjoint directed $ts$-paths $Q_1, Q_2, \dots, Q_k$.
\end{proof}

\begin{lemma}[{\cite[Lemma 10.7.9]{BJG09}}]\label{makecirc}
Let $D$ be an eulerian digraph, and suppose there is a directed trail visiting (not necessarily distinct) vertices $v_1, v_2, \dots, v_k$ in that order.  Then there is a directed euler circuit visiting $v_1, v_2, \dots, v_k$ in that order.
\end{lemma}

\begin{lemma}\label{interlace0mod4}
Let $D$ be an eulerian digraph, and $s,t \in V(D)$ with $s \ne t$ such that every edge cut of $\ug{D}$ separating $s$ and $t$ has at least $4$ edges.  Then $D$ has an euler circuit interlacing $s$ and $t$.  
\end{lemma}

\begin{proof}
If we follow the directed paths $P_1, Q_1, P_2, Q_2$ from Lemma \ref{kpaths} in that order, we obtain a circuit that interlaces $s$ and $t$.  Applying Lemma \ref{makecirc} extends this circuit to the desired euler circuit.
\end{proof}

Combining Theorem \ref{2badok} and Lemma \ref{interlace0mod4}, we obtain the following.

\begin{corollary}\label{4edgeconn-two0mod4}
Every eulerian digraph with exactly two vertices of degree $0 \mmod 4$ has some (necessarily orientable) bi-eulerian directed embedding, provided that every edge cut separating these two vertices has at least $4$ arcs.
\end{corollary}

\begin{corollary}\label{4edgeconn-two0mod4undir}
Every eulerian graph with exactly two vertices of degree $0 \mmod 4$ has some orientable bi-eulerian embedding, provided that every edge cut separating these two vertices has at least $4$ edges.
\end{corollary}

\subsection{Bi-eulerian embeddings with many vertices of degree $0$ mod $4$}
\label{ssec:many0mod4}

Theorem \ref{2badok} is an example of a more general principle.  The proof of Theorem \ref{2badok} uses the fact that there is a bi-eulerian embedding of $\dip_4$, and that the given euler circuit $T$ must visit the two vertices $x, y$ of degree $0$ mod $4$ in $D$ in the same order that the two vertices of $\dip_4$ are visited by one of the euler circuits in the bi-eulerian embedding of $\dip_4$, i.e., $x$ and $y$ are interlaced by $T$. 
Instead of $\dip_4$, we can use any 4-regular digraph $H$ having a bi-eulerian embedding.  The $0$ mod $4$ vertices of $D$, instead of just being interlaced, must be visited in an order given by one of the euler circuits in the bi-eulerian embedding of $H$. 

We begin by defining a generalization of interlacement. 
 
\begin{definition} Suppose $\tau$ is a sequence of length $n$ on $m$ symbols.  We say that a sequence $\sigma$ has a subsequence $\gamma$  that is \emph{patterned by $\tau$} if $\gamma$ also has length $n$ and $m$ symbols and if there exists a bijection $g$ from the symbols of $\tau$ to the symbols of $\gamma$ so that $g(\tau) = \gamma$, where $g(\tau)$ is the sequence that results from applying $g$ to the symbol in each position of $\tau$.  We can also apply this definition to cyclic sequences.
\end{definition}

This definition generalizes interlacing in that two vertices $v,w$ are interlaced by a circuit $C$ if the cyclic sequence of vertices of $C$ has a subsequence containing only $v$'s and $w$'s that is patterned by $(xyxy)$.

\begin{theorem}\label{2nbadok}
Let $D$ be an eulerian digraph in which the set $S$ of vertices with degree $0$ mod $4$ has $|S|=2n$.  Suppose there exists a $4$-regular digraph $H$ with $2n$ vertices and a bi-eulerian directed embedding with faces bounded by two euler circuits $X_1$ and $X_2$.  Let $\xi$ be any one of the cyclic sequences of vertices on $X_1$ (or its reverse), or $X_2$ (or its reverse).

If $T$ is a directed euler circuit of $D$ whose vertex sequence has a subsequence $\gamma$ that contains each vertex of $S$ exactly twice (and no other entries) and is patterned by at least one of the choices for $\xi$, then $D$ has a (necessarily orientable) bi-eulerian directed embedding with one of the faces bounded by $T$.
\end{theorem}

The proof is a direct adaptation of the proof of Theorem \ref{2badok} and is left to the reader.
The following example is an application of Theorem \ref{2nbadok} where the patterns can be described in terms of multiple pairs of interlaced vertices.

\begin{example}
Suppose we construct a digraph $H$ by starting with $\dip_4$ and taking repeated $2$-edge-joins of our current digraph and a copy of $\dip_4$.
Then $H$ is a $4$-regular digraph, and has an orientable bi-eulerian directed embedding by Observation \ref{2ecreduction}.  The patterns $\xi$ that can be obtained from bi-eulerian directed embeddings of all such digraphs $H$ provide the following consequence of Theorem \ref{2nbadok}.

Suppose an eulerian digraph $D$ has exactly $2t$ vertices $x_1, y_1, x_2, y_2, \dots x_t, y_t$ of degree $0$ mod $4$ (and possibly other vertices of degree $2$ mod $4$), 
and we can find two occurrences $v', v''$ of each $v \in \set{x_1, y_1, \dots, x_t, y_t}$ along an euler circuit $T$ with the following properties.
\begin{enumerate}[(1),nosep]

\item For each $i$ with $1 \le i \le t$, $x_i', y_i', x_i'', y_i''$ occur in that order along $T$. (The vertices of degree $0$ mod $4$ occur in pairs interlaced by $T$.)

\item For each $i, j$ with $1 \le i < j \le t$, all of $x_i', y_i', x_i'', y_i''$ appear in the subtrail of $T$ from $y_j''$ to $x_j'$.  (Loosely, each interlaced pair is independent of previous pairs.)

\end{enumerate}
Then $D$ has an orientable bi-eulerian embedding with $T$ as a face.  We omit the details of the proof.

Figure \ref{cyclicorder} shows an example of a cyclic ordering of vertices that satisfies conditions (1) and (2), and the digraph $H$ from which the corresponding pattern comes.
\end{example}

\begin{figure}[h!]
 \centering
  \hbox to \hsize{\hfil
    \includegraphics[scale=1.2]{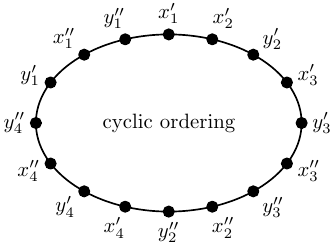}\hfil\hfil
    \raise10pt\hbox{\includegraphics[scale=1.2]{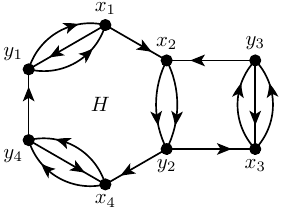}}\hfil
  }
\caption{Cyclic ordering and the digraph $H$ from which it comes.} 
 \label{cyclicorder}
\end{figure} 

\section{Orientable embeddings from circuit decompositions}\label{sec:circdec}

Kotzig's result, Theorem \ref{kot-bieul}, tells us that in a connected $4$-regular graph we can complete any circuit decomposition, not just an euler circuit, to an embedding by adding a new euler circuit, albeit without specifiying orientability.  In Section \ref{sec:no-eulf} below we will show that this is true for all eulerian graphs, and that in most cases we can guarantee that the embedding is nonorientable.  However, in this section we examine the orientable situation, where completing a circuit decomposition $\cC$ to an orientable embedding with an euler circuit (or more generally, with at most two circuits) gives an upper embedding relative to $\cC$.  We can also consider upper directed embeddings relative to directed circuit decompositions in eulerian digraphs, where we construct a directed embedding by adding one or two directed circuits.  We provide more details of some results referred to earlier, and show that a natural extension of the results of Section \ref{sec:suff} fails.

The only results that we know of regarding upper embeddings of circuit decompositions that are not just euler circuits appear in \cite{EGS20, GGS05, GMS20, GPS18}, mentioned in Section \ref{sec:intro}.  These investigate the completion of a triangular decomposition of a graph or digraph to an orientable embedding by adding an euler circuit, giving an upper relative embedding.  They consider triangular decompositions that arise from structures in design theory such as Steiner triple systems and latin squares.

Theorems \ref{dirsteiner} and \ref{dirlatin} were proved first for embeddings of undirected graphs, but were later strengthened to apply to directed embeddings of digraphs.  We state only the directed versions.

\begin{theorem}[Grannell, Griggs, and \siran. \cite{GGS05};  Griggs, McCourt, and \siran. {\cite[Theorem 1.1]{GMS20}}]
\label{dirsteiner}
Let $\cC$ be an oriented Steiner triple system, i.e., a decomposition of a regular tournament $D$ into directed triangles.  Then there is an orientable directed embedding of $D$ with the elements of $\cC$ as the profaces and with exactly one antiface, which is a directed euler circuit.
\end{theorem}

\begin{theorem}[Griggs, Psomas, and \siran. \cite{GPS18}; Griggs, McCourt, and \siran. {\cite[Theorem 1.2]{GMS20}}]
\label{dirlatin}
Let $\cC$ be an oriented latin square of odd order, i.e., a decomposition of an eulerian orientation $D$ of a balanced complete tripartite graph $K_{p,p,p}$ into directed triangles.  Then there is an orientable directed embedding of $D$ with the elements of $\cC$ as the profaces and with exactly one antiface, which is a directed euler circuit.
\end{theorem}

For graphs and digraphs obtained from the incidence structures known as symmetric configurations of block size $3$, there is a positive result for small cases but a negative result in general.

\begin{theorem}[Erskine, Griggs, and \siran. {\cite[Theorems 2.2 and 3.2]{EGS20}}]
\label{configuration}
Suppose $G$ is the associated graph of a symmetric configuration $n_3$, i.e., $G$ is a $6$-regular $n$-vertex simple graph with a decomposition $\cT$ into triangles.  Let $\cC$ be a collection of directed triangles obtained by directing the elements of $\cT$, and let $D$ be the digraph obtained from $G$ by directing the edges as specified by $\cC$.

\begin{enumerate}[(a), nosep]
\item\label{configyes}
If $n$ is odd and $7 \le n \le 19$, then every for every such $D$ and $\cC$ there is an orientable directed embedding of $D$ with the elements of $\cC$ as the profaces and with exactly one antiface, which is a directed euler circuit.

\item\label{configno}
If $n$ is odd and $n \ge 21$ then there exists such a graph $G=G_0$ with triangle decomposition $\cT=\cT_0$ such that there is no orientable embedding of $G_0$ with the elements of $\cT_0$ as faces and exactly one additional face.  Thus, for every $D$ and $\cC$ derived as above from $G_0$ and $\cT_0$ there is no orientable directed embedding of $D$ with the elements of $\cC$ as the profaces and with exactly one antiface.
\end{enumerate}
\end{theorem}

The proofs of Theorems \ref{dirsteiner} through \ref{configuration} use arguments that are specific to triangle decompositions.  They find embeddings of the bipartite \emph{Levi graph} or \emph{incidence graph} which represents incidences between triangles and vertices, and use a standard translation of those into embeddings of the original graph (see \cite{GMS20} for details).

It is natural to ask whether we can generalize Theorem \ref{2mod4main} or its undirected counterpart to say that when all degrees are $2$ mod $4$ we can complete a given circuit decomposition, rather than a just given euler circuit, to an orientable embedding by adding an euler circuit.
But Theorem \ref{configuration}\ref{configno} above shows that this is false.  The examples it gives are $6$-regular (so all vertices have degree $2$ mod $4$), $4$-edge-connected (so there is no problem created by $2$-edge-cuts), and a hypothetical relative embedding with an euler circuit face would have even Euler genus (so there is no parity issue with finding an orientable embedding), but we cannot complete the triangle decomposition to an embedding by adding an euler circuit.
These examples are simple graphs with at least $21$ vertices, but the same method (i.e., using incidence graphs) also gives smaller examples if we allow loops and multiple edges.
 % Note just for us: that EGS examples are $4$-edge-connected follows from fact that Levi graph is $2$-edge-connected because it is a cubic bipartite graph, and parity conditions imply that they have an odd number of faces other than triangle decomposition faces, i.e. `no one face' doesn't mean there are $2$ just for parity reasons.

For example, at left in Figure \ref{egsextendedfig} is a $4$-edge-connected $6$-regular $3$-vertex graph $G$ with a decomposition into three eulerian subgraphs with edge sets $\set{e_i, f_i, g_i}$ for $i = 1, 2, 3$. Let $\cC=\set{C_1, C_2, C_3}$ be a circuit decomposition of $G$ where each circuit traverses one of the subgraphs.  At right in Figure \ref{egsextendedfig} is the incidence graph $H$, a bipartite graph (with multiple edges) representing incidences between $V(G)$ (solid vertices) and $\cC$ (open vertices).
Under a slight generalization of the translation described in \cite{GMS20}, orientable embeddings of $H$ with $k$ (necessarily odd) faces correspond to orientable embeddings of $G$ relative to $\cC$ with $k$ outer faces.
But for every spanning tree $T$ of $H$, $H - E(T)$ has two components with an odd number of edges, so by Jungerman's condition \cite[Theorem 2]{Jun78} (or Xuong's maximum genus formula \cite{Xuo79b}) there is no orientable embedding of $H$ with one face, and hence no orientable embedding of $G$ relative to $\cC$ with an euler circuit outer face.

\begin{figure}
    \centering
    \hbox to \hsize{\hfil
        \includegraphics[scale=1.2]{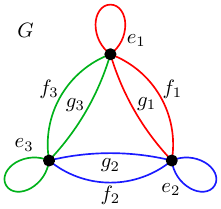}\hfil
        \includegraphics[scale=1.2]{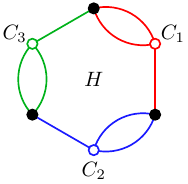}\hfil
    }
 \caption{Circuit decomposition that cannot be completed.}
 \label{egsextendedfig}
\end{figure}

%\section {Comments on the $4$-regular case}\label{sec:4reg} saved in DenseGraphs paper

\section{Nonorientable embeddings with euler circuit faces}
\label{sec:no-eulf}

A directed embedding with at least one euler circuit face is orientable by Lemma \ref{diremb-ori-2fc}, so for nonorientable embeddings with euler circuit faces we only study undirected graphs.  In this section we show that every euler circuit, or more generally every circuit decomposition, of an eulerian graph can be completed to an embedding by an euler circuit, and the embedding can be chosen to be nonorientable except in some well-characterized cases.

We use several results of \siran. and \skoviera. \cite{SS88}.
A collection of closed walks $\cW$ in a graph $G$ is \emph{realizable}, \emph{orientably realizable}, or \emph{nonorientably realizable} if $G$ has respectively an unrestricted, orientable, or nonorientable embedding relative to $\cW$.
\siran. and \skoviera. stated their results in terms of \emph{words}, sequences of half-edges incident with a vertex $v$.  In our framework their words correspond to components of the transition graph $\itr(\cW, v)$.

The first result we need characterizes realizable and orientably realizable collections $\cW$.

\begin{theorem}[{\siran. and \skoviera., \cite[Proposition 1 and Corollary 1]{SS88}}]\label{realizable}
Let $\cW$ be a collection of closed walks in a graph $G$.  Then $\cW$ is realizable if and only if $\itr(\cW, v)$ is a subgraph of a cycle (i.e., either the cycle itself or a union of disjoint paths) for every $v \in V(G)$.  Realizability implies that each edge of $G$ is used at most twice by $\cW$.

Moreover, a realizable $\cW$ is orientably realizable if and only if every element of $\cW$ can be oriented to give a directed walk, so that each edge used twice by $\cW$ is used once in each direction.
\end{theorem}

The second result that we use is a simplified statement of \siran. and \skoviera.'s result characterizing when $\cW$ is both orientably and nonorientably realizable \cite[Theorem 3]{SS88}.

\begin{proposition}\label{our-ori2nonori}
Let $\cW$ be an orientably realizable collection of closed walks in a graph $G$.  Then $\cW$ is also nonorientably realizable if and only if there exist a vertex $v$ and two half-edges incident with $v$ that belong to the same block of $G$ but different components of $\itr(\cW, v)$.
\end{proposition}

We also use the following result on reducing the number of faces in a relative embedding, and a simple observation about transition graphs.

\begin{theorem}[{\siran. and \skoviera., \cite[Theorem 1]{SS88}}]\label{ss-no-range}
Let $\cW$ be a realizable collection of walks in a graph $G$, and suppose that there is an embedding of $G$ relative to $\cW$ with $t$ outer faces.  Suppose the subgraph of $G$ obtained by deleting all the edges of $G$ that appear twice in $\cW$ has $c$ nontrivial components.  Then for every $s$ with $c \le s < t$ there is a nonorientable embedding of $G$ relative to $\cW$ with $s$ outer faces.
\end{theorem}

\begin{observation}\label{uniquewalks}
If a collection of closed walks $\cW$ uses each edge of $G$ at most once, then the graphs $\itr(\cW,v)$, $v \in V(G)$, indicate which edges of $G$ are used by $\cW$ and describe how to connect those edges together into closed walks.  They therefore determine $\cW$ uniquely.
\end{observation}

We can now state the main result of this section.  A \emph{tree of cycles} is a connected graph all of whose blocks are cycles.

\begin{theorem}\label{no-circdecomp}
Let $\cC$ be a decomposition of a nontrivial eulerian graph $G$ into circuits.
\begin{enumerate}[(a), nosep]
\item If $G$ is a tree of cycles and $\cC$ is the collection of cycles in $G$ then every embedding of $G$ relative to $\cC$ is planar and has exactly one outer face, and at least one such embedding exists.

\item Otherwise, $\cC$ is both orientably and nonorientably realizable.  There is at least one embedding of $G$ relative to $\cC$ that is nonorientable and has exactly one outer face.
\end{enumerate}
For every embedding of $G$ relative to $\cC$ with exactly one outer face, the outer face is bounded by an euler circuit, and the embedding has maximum Euler genus (and so maximum orientable or nonorientable genus, as appropriate) over all embeddings of $G$ relative to $\cC$.
\end{theorem}

\begin{proof}
Suppose $G$ is a nontrivial eulerian graph with a circuit decomposition $\cC$.  Since $G$ is nontrivial, every vertex has degree at least $2$.  Since $\cC$ uses every edge exactly once, $\itr(\cC, v)$ is a $1$-regular graph (perfect matching) on vertex set $E^*(v)$ for each $v \in V(G)$, and $\cC$ is orientably realizable by Theorem \ref{realizable}.

Suppose $G$ is a tree of cycles with $n$ vertices and $m$ edges, and $\cC$ is the collection of cycles of $G$, with $|\cC| = k$.
A simple proof by induction on $k$ shows that $m = n + k - 1$.
Suppose we have an embedding of $G$ relative to $\cC$ with $\ell$ outer faces. The Euler genus is $\gamma = 2-n+m-(k+\ell) = 1-\ell$.  Since $\gamma \ge 0$ and $\ell \ge 1$, we must have $\gamma=0$ and $\ell = 1$, i.e. the embedding is planar and there is exactly one outer face.

Now suppose $\cC$ is nonorientably realizable.
No edges are used twice by $\cC$, so if we apply Theorem \ref{ss-no-range}, we do not delete any edges and we obtain a single nontrivial component, $G$ itself. Therefore, there is a nonorientable relative embedding with exactly one outer face.

We now show that the previous two paragraphs cover all possibilities. Suppose $\cC$ is not nonorientably realizable.
Then, by Proposition \ref{our-ori2nonori}, at every vertex $v$ all half-edges incident with $v$ and in the same block of $G$ must be in the same component of $\itr(\cC, v)$.  Since $\itr(\cC, v)$ is a perfect matching, and each block containing $v$ has at least two half-edges incident with $v$, this is only possible if there are exactly two half-edges incident with $v$ from each such block, and some element of $\cC$ uses those two half-edges consecutively.  Therefore each block is $2$-regular, so each block is a cycle, and $G$ is a tree of cycles.
Moreover, if $\cZ$ is the collection of cycles of $G$, then $\itr(\cC, v) = \itr(\cZ, v)$ for all $v \in V(G)$, and so by Observation \ref{uniquewalks}, $\cC = \cZ$.  Thus, $G$ is a tree of cycles and $\cC$ is its collection of cycles, as required.

If an embedding of $G$ relative to $\cC$ has exactly one outer face, the facial walk must cover every edge exactly once, and so is an euler circuit.  The embedding has the minimum possible number of outer faces, and minimizing the number of faces maximizes the Euler genus.
\end{proof}

By taking $\cC = \{T\}$ where $T$ is an euler circuit, we obtain the following strengthening of Edmonds' result \cite[p.~123]{Edm65} that every eulerian graph has a bi-eulerian embedding.  It allows us to specify that one of the faces is bounded by $T$, and guarantees a nonorientable embedding except when the graph is a cycle.

\begin{corollary}\label{no-bieul}
Suppose $G$ is a nontrivial $n$-vertex $m$-edge eulerian graph and $T$ is an euler circuit of $G$.  Then $G$ has a bi-eulerian embedding in which one of the faces is bounded by $T$, and there is a nonorientable embedding of this kind unless $G$ is a cycle.
These bi-eulerian embeddings have maximum Euler genus (namely $m-n$), and so maximum orientable or nonorientable genus, as appropriate, over all embeddings of $G$ with a face bounded by $T$.
\end{corollary}

\section{Maximum genus and interpolation for nonorientable directed embeddings}
\label{sec:no-diremb}

To maximize the genus of an embedding of a given graph or digraph, we try to minimize the number of faces.  By Lemma \ref{diremb-ori-2fc}, an orientable directed embedding of a nontrivial eulerian digraph must have at least two faces.  However, for a nonorientable directed embedding it is possible that there is only one face, and in this short section we show that in fact this can always be achieved.  The argument is essentially the same as for undirected graphs, with one simple additional observation.

For undirected graphs Edmonds \cite[p.~123]{Edm65} showed that every connected graph $G$ has an embedding with a single face.
 % This was also shown by Pisanski \cite{Pis78} (see \cite[Theorem 3.1]{FPR14}).
Ringel \cite[Theorems 11 and 13]{Rin77} showed that the $1$-face embedding can be chosen to be nonorientable unless $G$ is a tree. Xuong \cite[note added in proof]{Xuo79b} stated something similar.
Stahl \cite[Theorem 8]{Sta78} showed that if a connected graph has an embedding with $t$ faces, then it has a nonorientable embedding with $s$ faces for all $s$, $1 \le s < t$.
\fijavz., Pisanski, and Rus showed that eulerian graphs have particularly nice $1$-face embeddings, where the face boundary forms a `strong parallel trace', as follows.

\begin{theorem}[\fijavz., Pisanski, and Rus {\cite[Theorem 5.3]{FPR14}}]
\label{fpr-oneface-s}
A graph $G$ has a $1$-face embedding where the face traverses each edge twice in the same direction if and only if $G$ is eulerian.  (If $G$ is nontrivial the embedding must be nonorientable.)
\end{theorem}

If we have a graph embedding represented by a rotation system with edge signatures, then \emph{twisting the edge $e$} means changing the signature of $e$.
Similarly, we can twist an arc in an embedding of a digraph.
From Observation \ref{diremb-basic-s}(a) and a result of Stahl \cite[Lemma 6(a)]{Sta78}, we obtain the following.

\begin{observation}\label{diremb-twistarc-2face}
(a) Twisting an arc in a directed embedding does not affect the alternating property of half-arcs around a vertex, so the result is still a directed embedding.

(b) If $a$ is an arc in a directed embedding that belongs to the boundaries of two distinct faces $X$ and $Y$, then twisting $a$ gives a new nonorientable directed embedding in which $X$ and $Y$ are combined into a single face and all other faces are unchanged.
\end{observation}

The following theorem covers existence, genus interpolation, and maximum genus for nonorientable directed embeddings.

\begin{theorem}\label{dir-oneface-s}
Suppose $D$ is an $n$-vertex $m$-edge eulerian digraph.  If $D$ has a directed embedding with $t$ faces, then $D$ has a nonorientable directed embedding with $s$ faces for all $s$ with $1 \le s < t$.   Therefore, the genera for which $D$ has a nonorientable directed embedding form a (possibly empty) interval.

If $D$ is trivial then $D$ has exactly one directed embedding, which is planar with exactly one face.
If $D$ is nontrivial then $D$ has at least one $1$-face directed embedding, every such embedding is nonorientable, and every such embedding has maximum genus (namely $m-n+1$) over all nonorientable directed embeddings of $D$, and also maximum Euler genus over all directed embeddings of $D$.
\end{theorem}

\begin{proof} 
Consider a directed embedding of $D$ with $t$ faces.  If there is more than one face, then we can twist an arc separating two distinct faces to combine them as in Observation \ref{diremb-twistarc-2face}(b), obtaining a nonorientable directed embedding.
Repeating until we obtain a directed embedding with only one face shows that there are nonorientable directed embeddings with $s$ faces for $1 \le s < t$.

Let $t\sbmax$ be the maximum number of faces in any directed embedding of $D$.  By the previous paragraph, the set of $s$ for which $D$ has a nonorientable directed embedding with $s$ faces is one of the (integer, possibly empty) intervals $[1, t\sbmax-1]$ or $[1,t\sbmax]$.  Hence, by Euler's formula, the genera also form an interval.

From Observation \ref{diremb-basic-s}\ref{db-ori-s} and the first paragraph of this proof, there is always a $1$-face directed embedding of $D$.
If $D$ is trivial, then its only embedding is planar.
If $D$ is nontrivial, then the single facial walk traverses every arc twice in the same direction, and hence every such embedding is nonorientable.

Thus, the maximum nonorientable genus and maximum Euler genus occur when there is one face, and the genus follows from Euler's formula.
\end{proof}

Theorem \ref{dir-oneface-s} may be regarded as a directed embedding version of Ringel's result (existence of a nonorientable $1$-face embedding) and Stahl's result (interpolation for nonorientable embeddings) mentioned at the start of this section; our proof is similar to Stahl's.  Moreover, applying Theorem \ref{dir-oneface-s} to an arbitrary eulerian orientation of an eulerian graph yields \fijavz., Pisanksi, and Rus's result, Theorem \ref{fpr-oneface-s}, as a corollary.

As Theorem \ref{dir-oneface-s} indicates, there are no orientable $1$-face directed embeddings of a nontrivial eulerian digraph.  Thus, all maximum Euler genus directed embeddings of a nontrivial eulerian digraph are nonorientable.

\section{Conclusion and Future Directions}\label{sec:conc}

In this paper we have essentially resolved the questions of maximum genus directed embeddings and of bi-eulerian embeddings of undirected graphs (or more generally, embeddings where a circuit decomposition can be extended to an embedding using an euler circuit face) in the nonorientable setting.  In the orientable setting we have given some simple necessary conditions, proved existence of bi-eulerian embeddings (but not extensions of circuit decompositions to embeddings using an euler circuit face) when all vertices have degree $2$ mod $4$, and discussed situations where we do and do not have a bi-eulerian embedding in the presence of vertices of degree $0$ mod $4$.  But there is clearly much more to do in the orientable case.

While our results here depend on vertex degrees modulo $4$, the proof of Theorem \ref{tournament} on tournaments, due to Bonnington et al.~\cite{BCMM02} does not.  This inspired us to consider dense graphs, having high minimum degree, with no conditions on vertex degrees modulo $4$.  We can show that given an $n$-vertex eulerian digraph where the minimum degree of the underlying simple undirected graph is at least $(4n+2)/5$, there is an orientable directed embedding with faces bounded by a specified euler circuit (or even a specified circuit decomposition) and at most two other circuits.  An overview of the proof appears in \cite{EE-Mdense}; a full paper is in preparation.

We have no reason to believe that $(4n+2)/5$ represents a sharp bound; it is just the point at which our current techniques stop working.  Therefore, it is natural to ask what minimum degree can guarantee an orientable bi-eulerian embedding for an admissible graph or digraph (i.e., it is eulerian, has an even number of vertices of degree $0$ mod $4$ and no bad $2$-edge-cuts).
All the examples in this paper that are admissible but have no orientable bi-eulerian embedding or directed embedding have vertices of degree $4$.  So perhaps a minimum degree of at least $6$ is sufficient.  In fact, our examples also all have multiple edges (digons), so it is even possible that all simple admissible graphs with minimum degree at least $4$ have a bi-eulerian embedding.

We can also ask what minimum degree is required to extend a specified euler circuit, or specified circuit decomposition, to an embedding by adding at most two other faces.  As the examples from Section \ref{sec:circdec} show, a minimum degree of at least $6$ is not sufficient to extend a general circuit decomposition.

As noted earlier, orientable bi-eulerian embeddings are closely connected to maximum genus directed embeddings of eulerian digraphs.
For general undirected graphs there are a number of useful results on maximum orientable genus, such as the formulas of Xuong \cite{Xuo79b} and \nebesky. \cite{Neb81}, and the polynomial-time algorithm of Furst, Gross, and McGeoch \cite{FGM88}.  We would like to obtain similar formulas for maximum orientable directed genus, and determine whether algorithmic problems involving orientable bi-eulerian embeddings and maximum orientable directed genus can be solved in polynomial time, or are NP-hard.
So far the only progress in this direction is the result mentioned in Section \ref{sec:intro}, where a result of Geelen, Iwata, and Murota \cite{GIM03} provides a polynomial-time algorithm to determine whether a $4$-regular eulerian digraph has a bi-eulerian directed embedding.

\section*{Acknowledgment}

We thank Karen Collins for pointing out the configuration $F_{2,2}$ inside our example in Figure \ref{CounterEx}(b), which led us to the infinite family in Proposition \ref{forbconf}.
We also thank Bill Jackson for useful discussions regarding \cite{ABJ96, BJ00, GIM03, GI05}.

\begin{appendices}

\section{Appendix: Forbidden configurations for orientable bi-eulerian embeddings}\label{app:forbconf}

In this appendix we prove Proposition \ref{forbconf}, which says that certain configurations in an eulerian graph $G$ prevent $G$ from having an orientable bi-eulerian embedding.  Proposition \ref{forbconf} follows from Theorem \ref{forbconf-main} below.

Henceforth we suppose that $G$ is an eulerian graph with an orientable bi-eulerian embedding, with faces bounded by two euler circuits $T_1, T_2$.  We will examine the way in which $T_1$ and $T_2$ travel through subgraphs we call `chains of digons', and use this to show that certain configurations $F_{s,t}$ cannot occur in $G$.

Since the embedding is orientable, we can orient $T_1$ and $T_2$ so that each edge is used once in each direction.  For our arguments it is actually more convenient to reverse the orientation of $T_2$, so we have oriented euler circuits $\ot_1, \ot_2$ which use each edge in the same direction (in other words, $\ot_1$ and $\ot_2$ are faces of a directed embedding of an orientation of $G$).
At each vertex the transitions used by $\ot_1$ and by $\ot_2$ are disjoint.

\subsection*{Transits through subgraphs}

We will consider how $\ot_1$ and $\ot_2$ enter and leave various subgraphs, including individual vertices.  If $H$ is a subgraph of $G$ and $\ot_i$
contains a sequence of half-edges $h_1 h_2 \dots h_k$ where $h_1$ and $h_k$ are not half-edges of $H$ but all half-edges $h_i$ with $2 \le i \le k-1$ belong to $H$, then we say $h_1h_k$ is a \emph{transit of $\ot_i$ through $H$}.  If $H$ is a single vertex $v$ then a transit is just a transition of $\ot_i$ at $v$, directed according to the orientation of $\ot_i$.  The \emph{profile of $\ot_i$ through $H$} is the set of all transits of $\ot_i$ through $H$ and the \emph{joint profile of $H$} is the ordered pair $(\pi_1, \pi_2)$ where $\pi_i$ is the profile of $\ot_i$ through $H$.  The number of transits of $\ot_i$ through $H$ is just $|A|/2$, where $A$ is the set of half-edges of $G$ that are not in $H$ but are incident with vertices of $H$.

Henceforth we use $\ot$ to denote an arbitrary oriented euler circuit.  When discussing our particular oriented euler circuits $\ot_1$ and $\ot_2$ we adopt the convention that $\set{i, j} = \set{1, 2}$, so if $\ot_i$ is one of our oriented euler circuits, $\ot_j$ is the other.  

If we know $\ot_i$ uses particular transits through certain subgraphs we may be able to conclude that some sequence of edges in $\ot_i$ closes up into a cyclic sequence without using all edges of $G$.
We then say that $\ot_i$ has a \emph{subcircuit}, which provides a contradiction because $\ot_i$ is a circuit using all edges.

\subsection*{Transits through vertices}

\begin{argstep}\label{X0}
Consider possible profiles of $\ot$ through $v$, a vertex of degree $4$.
Each of the $4!=24$ permutations of the half-edges $h_1, h_2, h_3, h_4$ incident with $v$ corresponds to a profile by considering the first and last ordered pairs of $h_i$'s as transits.  Thus for example the permuted order $h_2, h_3, h_4, h_1$ becomes the profile $\{h_2 h_3, h_4 h_1 \}$.  However, the profile $\prof(h_2h_3, h_4h_1)$ is the same as $\prof(h_4h_1, h_2h_3)$, so there are only $12$ distinct profiles.

For each profile $\pi_i$ of $\ot_i$ through $v$, the profile $\pi_j$ of $\ot_j$ through $v$ is uniquely determined by the facts that in $\ot_j$ each half-edge must be used in same direction as in $\ot_i$, and that the transitions of $\ot_j$ are disjoint from those of $\ot_i$.
In particular, if $\pi_i = \prof(h_1 h_2, h_3 h_4)$ then $\pi_j = \prof(h_1 h_4, h_3 h_2)$.
\end{argstep}

\subsection*{Transits through chains of digons}

A \emph{chain of digons from $u$ to $v$ of length $\ell$} is a subgraph $C$ consisting of distinct vertices $u=v_0, v_1, v_2, \dots, v_\ell = v$, all of degree $4$ in $G$, and two parallel edges $e_p^1, e_p^2$ joining $v_{p-1}$ and $v_p$ for $1 \le p \le \ell$.
We call $u=v_0$ the \emph{rear vertex} and $v=v_\ell$ the \emph{front vertex}.
For each $p$ with $0 \le p \le \ell$ denote the half-edges incident with $v_p$ by $g_p^1, g_p^2, h_p^1, h_p^2$, such that for $1 \le p \le \ell$ and $\al \in \{1, 2\}$ we have $e_p^\al = \set{h_{p-1}^\al, g_p^\al}$.  See Figure \ref{chainlabel}.  The half-edges $g_0^\al$ and $h_\ell^\al$ for $\al \in \set{1,2}$ are the half-edges incident with $v_0$ and $v_\ell$ that do not belong to $C$.

\begin{figure}[h!]
\centering   
\includegraphics[width=0.45
\textwidth]%
    {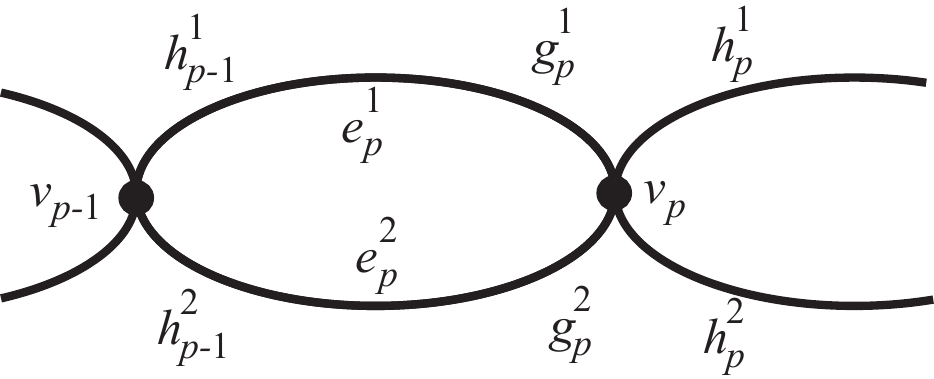}
%trim=left, bottom, right, top
\caption{Labels for a chain of digons.}
\label{chainlabel}
\end{figure}

At each vertex $v_p$ we split the profiles of $\ot$ through $v_p$ into the following cases, using the convention that $\set{\al,\be} = \set{\ga,\de} = \set{1,2}$.

$\ot$ \emph{turns around} at $v_p$ if the profile of $\ot$ through $v_p$ has the form
$\{g_p^\al g_p^\be, h_p^\ga h_p^\de\}$;

$\ot$ \emph{passes itself} at $v_p$ if the profile of $\ot$ through $v_p$ has the form
$\{g_p^\al h_p^\ga, h_p^\de g_p^\be\}$;

$\ot$ \emph{goes forward} at $v_p$ if the profile of $\ot$ through $v_p$ has the form
$\{g_p^1 h_p^\ga, g_p^2 h_p^\de\}$;

$\ot$ \emph{goes backward} at $v_p$ if the profile of $\ot$ through $v_p$ has the form
$\{h_p^\ga g_p^1, h_p^\de g_p^2\}$.

\noindent
The above four situations cover all $12$ possibilities.  See Figure \ref{CounterEx1} for examples.

\begin{figure}[ht!]
     \centering
     \begin{subfigure}[c]{0.75\textwidth}
        \centering
\includegraphics[ width=0.75\textwidth]{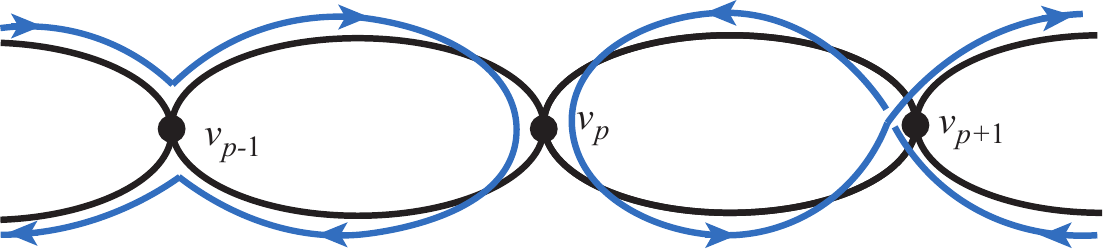}
        \caption{$\ot$ turning around at $v_p$ and passing itself at $v_{p-1}$ and $v_{p+1}$.}
        \label{BiDirTransit}
     \end{subfigure}
     \vspace{6mm}
     
             \begin{subfigure}[c]{0.45\textwidth}
        \centering    \includegraphics[width=0.85\textwidth]{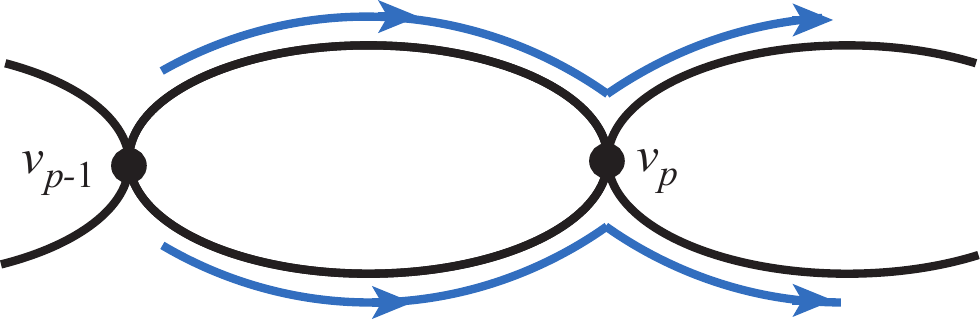}
        \caption{$\ot$ going forward and noncrossing at $v_p$.}
        \label{UniTransit}
     \end{subfigure}
    \hfill
         \begin{subfigure}[c]{0.45\textwidth}
        \centering
    \includegraphics[ width=0.85\textwidth]{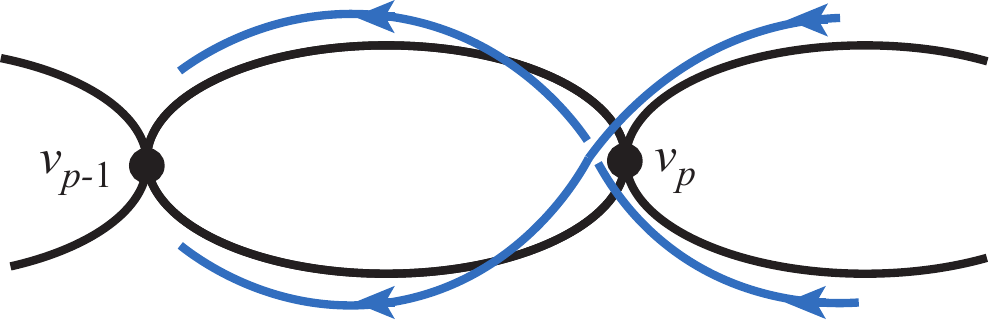}
        \caption{$\ot$ going backward and crossing at $v_p$.}
        \label{UniTransitCross}
     \end{subfigure}
  \caption{How $\protect\ot$ transits vertices in a chain of digons.}
\label{CounterEx1}
\end{figure}

\begin{argstep}\label{X1a}
If $\ot_i$ turns around or passes itself at $v_p$, we say $\ot_i$ is \emph{bidirectional} at $v_p$, while if $\ot_i$ goes forward or backward at $v_p$, we say $\ot_i$ is \emph{unidirectional} at $v_p$.
If $\ot_i$ is bidirectional at $v_p$ then it is bidirectional at $v_{p-1}$ (if it exists) and at $v_{p+1}$ (if it exists) and hence bidirectional at all vertices $v_q$, $0 \le q \le \ell$.  Similarly, if $\ot_i$ goes forward or backward at $v_p$, then it goes forward or backward, respectively, at all $v_q$, $0 \le q \le \ell$.

By (\ref{X0}), if $\ot_i$ turns around or passes itself at $v_p$, then $\ot_j$ passes itself or turns around at $v_p$, respectively.  And if $\ot_i$ goes forward or backward at $v_p$, then $\ot_j$ does the same thing.  It follows that $\ot_i$ and $\ot_j$ behave in the same way to the extent that they are both go forward at all vertices, both go backward at all vertices, or both are bidirectional at all vertices.  See Figure \ref{A2fig}.

\begin{figure}[ht!]
     \centering
     \begin{subfigure}[c]{0.28\textwidth}
        \centering    \includegraphics[width=0.7\textwidth]{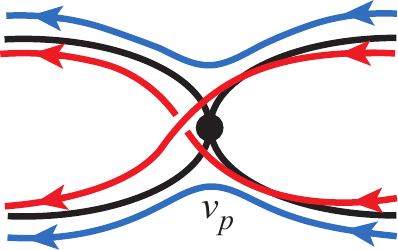}
        \caption{$\ot_i$ and $\ot_j$ both going backwards at $v_p$.}
        \label{A2c}
     \end{subfigure}
     \hfill
     \begin{subfigure}[c]{0.28\textwidth}
        \centering
\includegraphics[ width=0.7\textwidth]{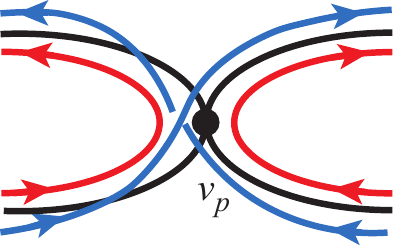}
        \caption{$\ot_i$ turning around and $\ot_j$ passing itself at $v_p$.}
        \label{A2a}
     \end{subfigure}
     \hfill
        \begin{subfigure}[c]{0.28\textwidth}
        \centering    \includegraphics[width=0.7\textwidth]{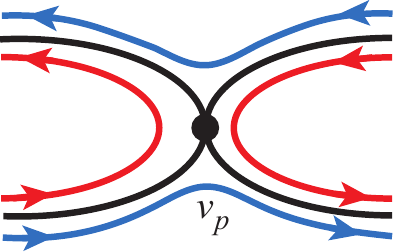}
        \caption{$\ot_i$ turning around and $\ot_j$ passing itself at $v_p$.}
        \label{A2b}
     \end{subfigure}
           \caption{$T_i$ and $T_j$ both go forwards,  backwards, or are bidirectional at $v_p$.}
    \label{A2fig}
    \end{figure}

\end{argstep}

\begin{argstep}\label{X1b}
Suppose both $\ot_1$ and $\ot_2$ either go forward or go backward at all vertices of $C$.  Then we say that $\ot_i$ is \emph{noncrossing} or \emph{crossing} at $v_p$ according to whether $\ga=1$ and $\de=2$ or $\ga=2$ and $\de=1$, respectively, in the definition above.  See Figure \ref{UniTransit} and \subref{UniTransitCross}.
By (\ref{X0}), $\ot_i$ is noncrossing at $v_p$ if and only if $\ot_j$ is crossing at $v_p$.
Thus, at each of the $\ell+1$ vertices $v_0, v_1, \dots, v_\ell$, exactly one of $\ot_i$ and $\ot_j$ is crossing.

If $\ell$ is odd, so that $\ell+1$ is even, the numbers of vertices at which $\ot_1$ and $\ot_2$ are crossing have the same parity.  Thus, the joint profile of $C$ is
$\jprof(g_0^1 h_\ell^\ga, g_0^2 h_\ell^\de;
	g_0^1 h_\ell^\ga, g_0^2 h_\ell^\de)$
where $\set{\ga, \de} = \set{1,2}$.
On the other hand, if $\ell$ is even, then the number of vertices at which $\ot_1$ and $\ot_2$ are crossing have opposite parities, so the joint profile of $C$ is
$\jprof(g_0^1 h_\ell^\ga, g_0^2 h_\ell^\de;
	g_0^1 h_\ell^\de, g_0^2 h_\ell^\ga)$
where $\set{\ga, \de} = \set{1,2}$.  See Figure~\ref{A3fig}.

\begin{figure}[ht!] 
     \centering
     \begin{subfigure}[c]{0.75\textwidth}
        \centering
\includegraphics[ width=0.9\textwidth]{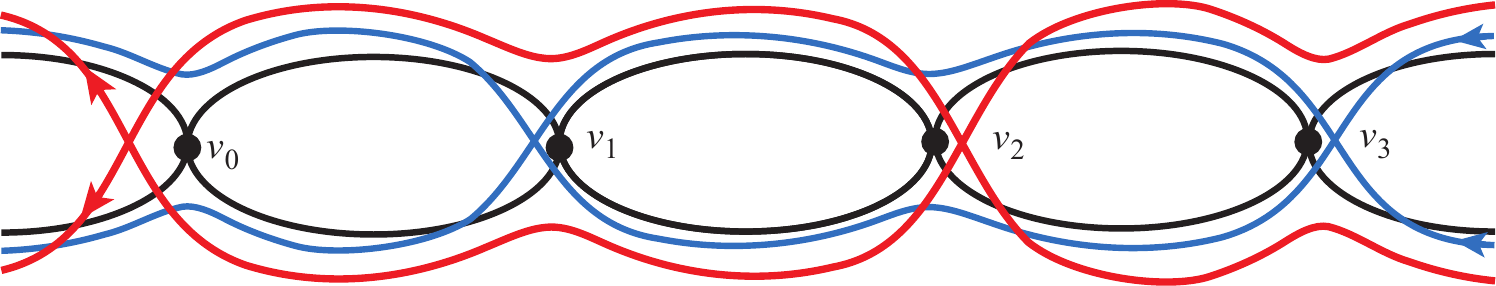}
        \caption{ $\ell$ is odd (here $\ell=3$), and the numbers of vertices at which $\ot_1$ and $\ot_2$ are crossing have the same parity.}
        \label{A3a}
     \end{subfigure}
     \vspace{6mm}
     
        \begin{subfigure}[c]{0.75\textwidth}
        \centering    \includegraphics[width=0.7\textwidth]{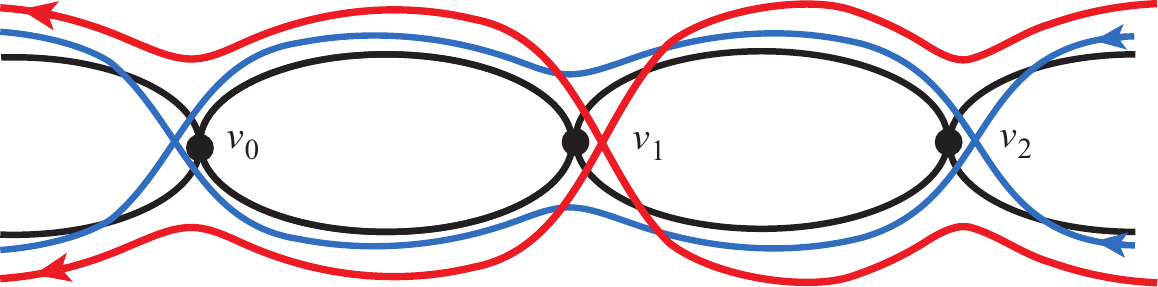}
        \caption{$\ell$ is even (here $\ell=2$), and the numbers of vertices at which $\ot_1$ and $\ot_2$ are crossing have opposite parity.}
        \label{A3b}
     \end{subfigure}
      \caption{Two joint profiles depending on whether $\ell$ is even or odd.}
      \label{A3fig}
           \end{figure}

\end{argstep}

\begin{argstep}\label{X2}
Suppose $\ot_1$ and $\ot_2$ are bidirectional.
Assume there are two vertices $v_p$ and $v_q$ where $\ot_1$ turns around. We may suppose that $p < q$ and that $\ot_1$ does not turn around at $v_r$ for $p < r < q$.  Then $\ot_1$ must travel from $v_p$ from $v_q$, then back to $v_p$, completing a subcircuit, which is impossible.
Therefore $\ot_1$ turns around at most once, and hence, by (\ref{X1a}), $\ot_2$ passes itself at most once.  Swapping the roles of $\ot_1$ and $\ot_2$, we conclude that each of $\ot_1$ and $\ot_2$ turns around once and passes itself once. Thus, $\ell=1$ and at $v_0$ some $\ot_i$ passes itself and $\ot_j$ turns around, while at $v_1$ we have that $\ot_i$ turns around and $\ot_j$ passes itself.   See Figure~\ref{A4fig}.

\begin{figure}[h!]
\centering   

\includegraphics[width=0.35
\textwidth]%
    {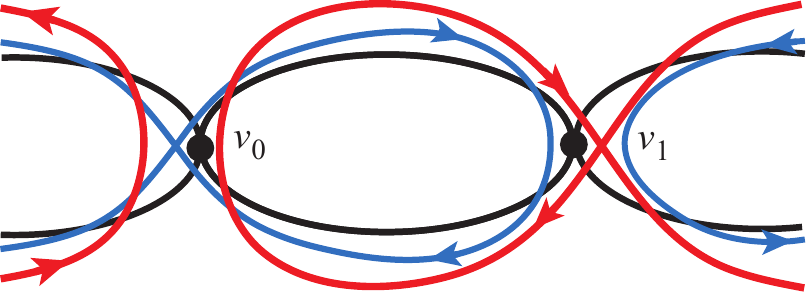}
%trim=left, bottom, right, top
\caption{Turn arounds that force $\ell$ to be 1.}
\label{A4fig}
\end{figure}

The sequences of half-edges used by $\ot_i$ as it passes through $C$ therefore have the form
$g_0^\al h_0^\ka g_1^\ka g_1^\la h_0^\la g_0^\be$
and $h_1^\ga h_1^\de$,
where $\set{\al, \be} = \set{\ga, \de} = \set{\ka, \la} = \set{1,2}$.
Applying (\ref{X0}), the sequences of half-edges used by $\ot_j$ as it passes through $C$ are then
$g_0^\al g_0^\be$ and
$h_1^\ga g_1^\la h_0^\la h_0^\ka g_1^\ka h_1^\de$.
In particular, both $\ot_i$ and $\ot_j$ have the same transit $h_1^\ga h_1^\de$ through $C$.
\end{argstep}

\subsection*{Forbidden configurations}

We now define certain configurations and prove that in many cases they prevent the existence of an orientable bi-eulerian embedding.

The configuration $F_{s,t}$ consists of two disjoint chains of digons $C$ and $\bC$ connected together as follows.
Assume $C$ has length $\ell=s$ and its elements are labelled as described above.  Assume $\bC$ has length $\bl=t$ and its elements are labelled so that $\bv_p, \bg_p^\al, \bh_p^\al$ play the same roles in $\bC$ as $v_p, g_p^\al, h_p^\al$, respectively, play in $C$.
We have a configuration $F_{s,t}$ if there are edges from the front vertex $v_s$ of $C$ to both the front and rear vertices $\bv_0$ and $\bv_t$ of $\bar C$.  In particular, we suppose $G$ has edges
$f_1 = \set{h_s^1, \bg_0^1}$ and
$f_2 = \set{h_s^2, \bh_t^1}$.

\begin{theorem}\label{forbconf-main}
If $G$ is an eulerian graph with an orientable bi-eulerian embedding containing a configuration $F_{s,t}$ with $s, t \ge 1$, then either $t=1$, or $s=1$ and $t$ is odd.

\end{theorem}

\begin{proof}
Suppose $G$ contains $F_{s,t}$, where $s, t \ge 1$, and $G$ has an orientable bi-eulerian embedding.  We assume the notation described above for the chains of digons $C, \bC$ in $F_{s,t}$ and for the oriented face boundaries $\ot_1, \ot_2$ of the embedding.
Without loss of generality we may assume that $\ot_1$ and $\ot_2$ use $f_1$ in the direction from $v_s$ to $\bv_0$ (from $h_s^1$ to $\bg_0^1$).

\begin{figure}[h!]
\centering   
\includegraphics[width=0.6\textwidth]
{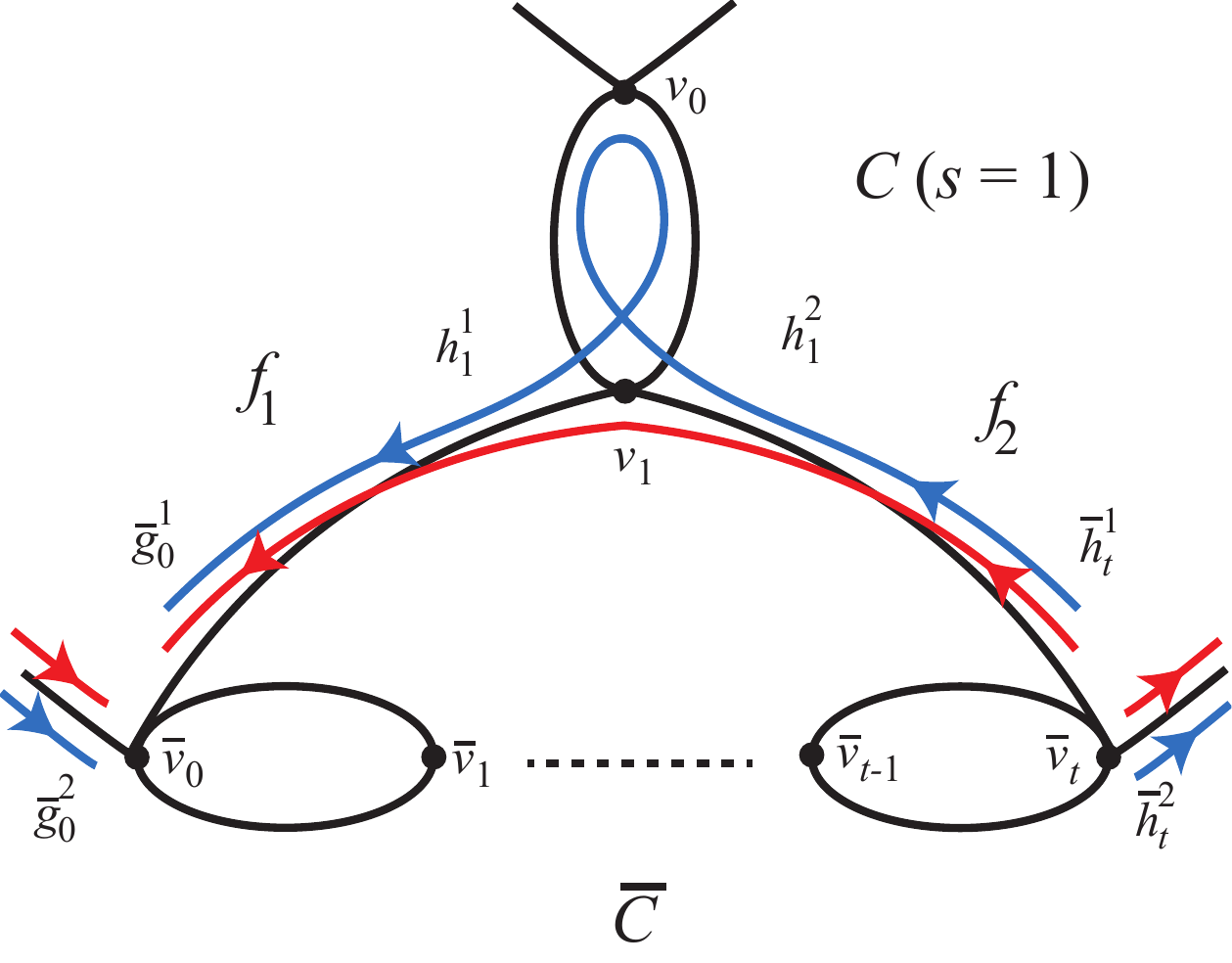}
\caption{{Paths through $F_{1,t}$}.}
\label{ThrmA5}
\end{figure}

If $t=1$ then the conclusion holds, so we may suppose that $t \ge 2$.  Then by (\ref{X2}) $\ot_1, \ot_2$ are not bidirectional with respect to $\bC$.
By (\ref{X1a}) and the fact that $\ot_1, \ot_2$ enter $\bC$ along $\bg_0^1$, $\ot_1$ and $\ot_2$ go forward at all vertices of $\bC$.  Therefore, $f_2$ must be used in the direction from $\bv_t$ to $v_s$ (from $\bh_t^1$ to $h_s^2$).  Thus, $\ot_1, \ot_2$ are bidirectional at $v_s$.
Therefore, by (\ref{X2}) we have $s=1$ and both $\ot_1$ and $\ot_2$ have a transit $h_s^2 h_s^1$ through $C$.  See Figure~\ref{ThrmA5}.

Suppose that $t$ is even.  Then by (\ref{X1b}) the joint profile of $\bC$ is
$\jprof(\bg_0^1 \bh_t^\ga, \bg_0^2 \bh_t^\de;
	\bg_0^1 \bh_t^\de, \bg_0^2 \bh_t^\ga)$
where $\set{\ga, \de} = \set{1,2}$.
If $\ga = 1$ and $\de = 2$ then $\ot_1$ has transits $\bg_0^1 \bh_t^1$ through $\bC$ and $h_s^2 h_s^1$ through $C$ which produce a subcircuit in $\ot_1$.  Similarly, if $\ga=2$ and $\de=1$ then $\ot_2$ has a subcircuit.  In either case we have a contradiction, so $t$ must be odd.

Thus, we have shown that either $t=1$, or $s=1$ and $t$ is odd, proving the result.
\end{proof}

Proposition \ref{forbconf} is just the contrapositive of Theorem \ref{forbconf-main}.

If $t=1$, or $s=1$ and $t$ is odd, then there exist eulerian graphs $G$ containing a configuration $F_{s,t}$ that also have an orientable bi-eulerian embedding.  We leave the verification of this to the reader: examples for $F_{2,1}, F_{3,1}$ and $F_{1,3}$ are not hard to construct and can easily be generalized to the other relevant $F_{s,t}$.

\end{appendices}

\printbibliography

@article {ABJ96,
    AUTHOR = {Andersen, Lars D\o{v}ling and Bouchet, Andr\'{e} and Jackson, Bill},
     TITLE = {Orthogonal {A}-trails of {$4$}-regular graphs embedded in
              surfaces of low genus},
   JOURNAL = {J. Combin. Theory Ser. B},
  FJOURNAL = {Journal of Combinatorial Theory. Series B},
    VOLUME = {66},
      YEAR = {1996},
    NUMBER = {2},
     PAGES = {232--246},
      ISSN = {0095-8956},
   MRCLASS = {05C10 (05C38 05C45)},
  MRNUMBER = {1376047},
MRREVIEWER = {Joan Hutchinson},
       DOI = {10.1006/jctb.1996.0017},
       URLX = {https://doi.org/10.1006/jctb.1996.0017},
}

@article {BCMM02,
    AUTHOR = {Bonnington, C. Paul and Conder, Marston and Morton, Margaret
              and McKenna, Patricia},
     TITLE = {Embedding digraphs on orientable surfaces},
   JOURNAL = {J. Combin. Theory Ser. B},
  FJOURNAL = {Journal of Combinatorial Theory. Series B},
    VOLUME = {85},
      YEAR = {2002},
    NUMBER = {1},
     PAGES = {1--20},
      ISSN = {0095-8956},
   MRCLASS = {05C10 (05C20)},
  MRNUMBER = {1900680},
MRREVIEWER = {Jozef \v{S}ir\'{a}\v{n}},
       DOI = {10.1006/jctb.2001.2085},
       URLX = {https://doi.org/10.1006/jctb.2001.2085},
}

@book {Fle90,
    AUTHOR = {Fleischner, Herbert},
     TITLE = {Eulerian graphs and related topics. {P}art 1. {V}ol. 1},
    SERIES = {Annals of Discrete Mathematics},
    VOLUME = {45},
 PUBLISHER = {North-Holland Publishing Co., Amsterdam},
      YEAR = {1990},
     PAGES = {xiv+387 pages},
      ISBN = {0-444-88395-9},
   MRCLASS = {05C45 (90B10)},
  MRNUMBER = {1055084},
MRREVIEWER = {Robin J. Wilson},
}

@article {FPR14,
    AUTHOR = {Fijav{\v z}, Ga\v{s}per and Pisanski, Toma\v{z} and Rus, Jernej},
     TITLE = {Strong traces model of self-assembly polypeptide structures},
   JOURNAL = {MATCH Commun. Math. Comput. Chem.},
  FJOURNAL = {MATCH. Communications in Mathematical and in Computer Chemistry},
    VOLUME = {71},
      YEAR = {2014},
    NUMBER = {1},
     PAGES = {199--212},
      ISSN = {0340-6253},
   MRCLASS = {05C90 (05C10)},
  MRNUMBER = {3184644},
}

@article {Edm65,
    AUTHOR = {Edmonds, Jack},
     TITLE = {On the surface duality of linear graphs},
   JOURNAL = {J. Res. Nat. Bur. Standards Sect. B},
  FJOURNAL = {Journal of Research of the National Bureau of Standards.
              Section B. Mathematics and Mathematical Physics},
    VOLUME = {69B},
      YEAR = {1965},
     PAGES = {121--123},
      ISSN = {0022-4340},
   MRCLASS = {55.10},
  MRNUMBER = {182962},
MRREVIEWER = {P. J. Higgins},
}

@book {MT,
    AUTHOR = {Mohar, Bojan and Thomassen, Carsten},
     TITLE = {{G}raphs on {S}urfaces},
    SERIES = {Johns Hopkins Studies in the Mathematical Sciences},
 PUBLISHER = {Johns Hopkins University Press, Baltimore, MD},
      YEAR = {2001},
     PAGES = {xii+291},
      ISBN = {0-8018-6689-8},
   MRCLASS = {05C10 (57M15)},
  MRNUMBER = {1844449},
MRREVIEWER = {Arthur T. White},
}

@book {E-MM13,
    AUTHOR = {Ellis-Monaghan, Joanna A. and Moffatt, Iain},
     TITLE = {{G}raphs on {S}urfaces: {D}ualities, {P}olynomials, and {K}nots},
    SERIES = {SpringerBriefs in Mathematics},
 PUBLISHER = {Springer, New York},
      YEAR = {2013},
     PAGES = {xii+139},
      ISBN = {978-1-4614-6970-4; 978-1-4614-6971-1},
   MRCLASS = {05-02 (05C10 05C31 57-01 57M25)},
  MRNUMBER = {3086663},
MRREVIEWER = {Lorenzo Traldi},
       DOI = {10.1007/978-1-4614-6971-1},
       ignore-URLX = {https://doi.org/10.1007/978-1-4614-6971-1},
}

@inproceedings {Kot68,
    AUTHOR = {Kotzig, A.},
     TITLE = {Eulerian lines in finite {$4$}-valent graphs and their
              transformations },
 BOOKTITLE = {Theory of {G}raphs ({P}roc. {C}olloq., {T}ihany, 1966)},
     PAGES = {219--230},
 PUBLISHER = {Academic Press, New York},
      YEAR = {1968},
   MRCLASS = {05.40},
  MRNUMBER = {0248043},
MRREVIEWER = {H. V. Kronk},
}

@book {GT,
    AUTHOR = {Gross, Jonathan L. and Tucker, Thomas W.},
     TITLE = {{T}opological {G}raph {T}heory},
      NOTE = {Reprint of the 1987 original [Wiley, New York;  MR0898434
              (88h:05034)] with a new preface and supplementary
              bibliography},
 PUBLISHER = {Dover Publications, Inc. Mineola, NY},
      YEAR = {2001},
     PAGES = {xvi+361},
      ISBN = {0-486-41741-7},
   MRCLASS = {05C10 (20F65 57M15)},
  MRNUMBER = {1855951},
}

@article {Xuo79b,
    AUTHOR = {Nguyen Huy Xuong},
     TITLE = {How to determine the maximum genus of a graph},
   JOURNAL = {J. Combin. Theory Ser. B},
  FJOURNAL = {Journal of Combinatorial Theory. Series B},
    VOLUME = {26},
      YEAR = {1979},
    NUMBER = {2},
     PAGES = {217--225},
      ISSN = {0095-8956},
   MRCLASS = {05C10},
  MRNUMBER = {532589},
MRREVIEWER = {F. Harary},
       DOI = {10.1016/0095-8956(79)90058-3},
       URLX = {https://doi-org.proxy.uba.uva.nl/10.1016/0095-8956(79)90058-3},
}

@article {EE-M19,
    AUTHOR = {Ellingham, M. N. and Ellis-Monaghan, Joanna A.},
     TITLE = {Edge-outer graph embedding and the complexity of the {DNA}
              reporter strand problem},
   JOURNAL = {Theoret. Comput. Sci.},
  FJOURNAL = {Theoretical Computer Science},
    VOLUME = {785},
      YEAR = {2019},
     PAGES = {117--127},
      ISSN = {0304-3975},
   MRCLASS = {05C10 (68Q05 68Q17 92D20)},
  MRNUMBER = {3986448},
       DOI = {10.1016/j.tcs.2019.03.019},
       URLX = {https://doi-org.proxy.uba.uva.nl/10.1016/j.tcs.2019.03.019},
}

@article {JSW09,
    AUTHOR = {Jonoska, Nata{\v{s}}a and Seeman, Nadrian C. and Wu, Gang},
     TITLE = {On existence of reporter strands in {DNA}-based graph
              structures},
   JOURNAL = {Theoret. Comput. Sci.},
  FJOURNAL = {Theoretical Computer Science},
    VOLUME = {410},
      YEAR = {2009},
    NUMBER = {15},
     PAGES = {1448--1460},
      ISSN = {0304-3975},
     CODEN = {TCSDI},
   MRCLASS = {57M15 (05C10 05C90 68R10 92D10 92E10)},
  MRNUMBER = {2499955 (2010f:57007)},
MRREVIEWER = {Yaokun Wu},       
}

@book {West,
    AUTHOR = {West, Douglas B.},
     TITLE = {{I}ntroduction to {G}raph {T}heory},
 PUBLISHER = {{Prentice Hall, Inc., Upper Saddle River, NJ}},
      YEAR = {1996},
     PAGES = {xvi+512},
      ISBN = {0-13-227828-6},
   MRCLASS = {05-01},
  MRNUMBER = {1367739},
}

@article {ESZtripnon,
    AUTHOR = {Ellingham, M. N. and Stephens, Chris and Zha, Xiaoya},
     TITLE = {The nonorientable genus of complete tripartite graphs},
   JOURNAL = {J. Combin. Theory Ser. B},
  FJOURNAL = {Journal of Combinatorial Theory. Series B},
    VOLUME = {96},
      YEAR = {2006},
    NUMBER = {4},
     PAGES = {529--559},
      ISSN = {0095-8956},
   MRCLASS = {05C10},
  MRNUMBER = {2232390},
MRREVIEWER = {Arthur T. White},
       DOI = {10.1016/j.jctb.2005.10.004},
       URLX = {https://doi-org.proxy.uba.uva.nl/10.1016/j.jctb.2005.10.004},
}

@article {H18,
    AUTHOR = {Hao, Rong-Xia},
     TITLE = {A note on the directed genus of {$K_{n,n,n}$} and {$K_n$}},
   JOURNAL = {Ars Math. Contemp.},
  FJOURNAL = {Ars Mathematica Contemporanea},
    VOLUME = {14},
      YEAR = {2018},
    NUMBER = {2},
     PAGES = {375--385},
      ISSN = {1855-3966},
   MRCLASS = {05C10 (05B05 05B07)},
  MRNUMBER = {3812027},
MRREVIEWER = {Constantinos Psomas},
       DOI = {10.26493/1855-3974.911.3b4},
       URLX = {https://doi-org.proxy.uba.uva.nl/10.26493/1855-3974.911.3b4},
}

@inproceedings {SS87,
    MYNOTE = {Removed stray "(1988)" from PAGES.  Added journal name also as SERIES since @inproceedings doesn't print journal.},
    AUTHOR = {{\v S}koviera, M. and {\v S}ir\'{a}\v{n}, J.},
     TITLE = {Oriented relative embeddings of graphs},
 BOOKTITLE = {Proceedings of the {I}nternational {C}onference on
              {C}ombinatorial {A}nalysis and its {A}pplications
              ({P}okrzywna, 1985)},
   JOURNAL = {Zastos. Mat.},
  FJOURNAL = {Polska Akademia Nauk. Instytut Matematyczny. Zastosowania
              Matematyki},
    SERIES = {Zastos Mat.},
    VOLUME = {19},
      YEAR = {1987},
    NUMBER = {3-4},
     PAGES = {589--597},
      ISSN = {0044-1899},
   MRCLASS = {05C10},
  MRNUMBER = {951374},
MRREVIEWER = {M. L. Marx},
}

@incollection {EW08,
    AUTHOR = {Ellingham, M. N. and Weaver, Adam},
     TITLE = {Constructing all minimum genus embeddings of {$K_{3,n}$}},
 BOOKTITLE = {The {I}nternational {C}onference on {T}opological and
              {G}eometric {G}raph {T}heory},
    SERIES = {Electron. Notes Discrete Math.},
    VOLUME = {31},
     PAGES = {293--298},
 PUBLISHER = {Elsevier Sci. B. V., Amsterdam},
      YEAR = {2008},
   MRCLASS = {05C10},
  MRNUMBER = {2571152},
       DOI = {10.1016/j.endm.2008.06.059},
       URLX = {https://doi-org.proxy.uba.uva.nl/10.1016/j.endm.2008.06.059},
}

@article {GGS05,
    AUTHOR = {Grannell, Mike J. and Griggs, Terry S. and {\v S}ir\'{a}\v{n}, Jozef},
     TITLE = {Maximum genus embeddings of {S}teiner triple systems},
   JOURNAL = {European J. Combin.},
  FJOURNAL = {European Journal of Combinatorics},
    VOLUME = {26},
      YEAR = {2005},
    NUMBER = {3-4},
     PAGES = {401--416},
      ISSN = {0195-6698},
   MRCLASS = {05B07 (05C10)},
  MRNUMBER = {2116179},
MRREVIEWER = {David E. Woolbright},
       DOI = {10.1016/j.ejc.2004.01.014},
       URLX = {https://doi-org.proxy.uba.uva.nl/10.1016/j.ejc.2004.01.014},
}

@article {GMS20,
    AUTHOR = {Griggs, Terry S. and McCourt, Thomas A. and {\v S}ir\'{a}\v{n}, Jozef},
     TITLE = {On the upper embedding of {S}teiner triple systems and {L}atin
              squares},
   JOURNAL = {Ars Math. Contemp.},
  FJOURNAL = {Ars Mathematica Contemporanea},
    VOLUME = {18},
      YEAR = {2020},
    NUMBER = {1},
     PAGES = {127--135},
      ISSN = {1855-3966},
   MRCLASS = {05B07 (05B15 05C10)},
  MRNUMBER = {4154732},
MRREVIEWER = {Shaopu Zhang},
       DOI = {10.26493/1855-3974.1959.9c7},
       URLX = {https://doi-org.proxy.uba.uva.nl/10.26493/1855-3974.1959.9c7},
}

@article {GPS18,
    AUTHOR = {Griggs, Terry S. and Psomas, Constantinos and {\v S}ir\'{a}\v{n}, Jozef},
     TITLE = {Maximum genus embeddings of {L}atin squares},
   JOURNAL = {Glasg. Math. J.},
  FJOURNAL = {Glasgow Mathematical Journal},
    VOLUME = {60},
      YEAR = {2018},
    NUMBER = {2},
     PAGES = {495--504},
      ISSN = {0017-0895},
   MRCLASS = {05B15 (05C10)},
  MRNUMBER = {3784061},
MRREVIEWER = {Atif A. Abueida},
       DOI = {10.1017/S0017089517000234},
       URLX = {https://doi-org.proxy.uba.uva.nl/10.1017/S0017089517000234},
}

@article {EGS20,
    AUTHOR = {Erskine, Grahame and Griggs, Terry and {\v S}ir\'{a}\v{n}, Jozef},
     TITLE = {On the upper embedding of symmetric configurations with block
              size 3},
   JOURNAL = {Discrete Math.},
  FJOURNAL = {Discrete Mathematics},
    VOLUME = {343},
      YEAR = {2020},
    NUMBER = {4},
     PAGES = {paper 111774 (6 pages)},
      ISSN = {0012-365X},
   MRCLASS = {05C10 (05B07)},
  MRNUMBER = {4044632},
       DOI = {10.1016/j.disc.2019.111774},
       URLX = {https://doi-org.proxy.uba.uva.nl/10.1016/j.disc.2019.111774},
}

@incollection{EE-Mdense,
booktitle={Proceedings of the 12th European Conference on Combinatorics, Graph Theory and Applications (Prague, 2023)},
title={Maximum genus orientable embeddings from circuit decompositions of dense eulerian graphs and digraphs},
author={Ellingham, M. N. and Ellis-Monaghan, Joanna. A.},
URL={https://doi.org/10.5817/CZ.MUNI.EUROCOMB23-056},
pages={401--406},
year={2023},
publisher={Masaryk University Press},
}

@article {Rin77,
    AUTHOR = {Ringel, Gerhard},
     TITLE = {The combinatorial map color theorem},
   JOURNAL = {J. Graph Theory},
  FJOURNAL = {Journal of Graph Theory},
    VOLUME = {1},
      YEAR = {1977},
    NUMBER = {2},
     PAGES = {141--155},
      ISSN = {0364-9024},
   MRCLASS = {05C15},
  MRNUMBER = {444509},
MRREVIEWER = {Arthur T. White},
       DOI = {10.1002/jgt.3190010210},
       URLX = {https://doi-org.proxy.uba.uva.nl/10.1002/jgt.3190010210},
}

@article {Sta78,
    AUTHOR = {Stahl, Saul},
     TITLE = {Generalized embedding schemes},
   JOURNAL = {J. Graph Theory},
  FJOURNAL = {Journal of Graph Theory},
    VOLUME = {2},
      YEAR = {1978},
    NUMBER = {1},
     PAGES = {41--52},
      ISSN = {0364-9024},
   MRCLASS = {05C10},
  MRNUMBER = {485488},
MRREVIEWER = {Seth R. Alpert},
       DOI = {10.1002/jgt.3190020106},
       URLX = {https://doi-org.proxy.uba.uva.nl/10.1002/jgt.3190020106},
}

@article {SS88,
    AUTHOR = {{\v S}ir\'{a}\v{n}, Jozef and {\v S}koviera, Martin},
     TITLE = {Relative embeddings of graphs on closed surfaces},
   JOURNAL = {Math. Nachr.},
  FJOURNAL = {Mathematische Nachrichten},
    VOLUME = {136},
      YEAR = {1988},
     PAGES = {275--284},
      ISSN = {0025-584X},
   MRCLASS = {05C10},
  MRNUMBER = {952479},
MRREVIEWER = {Arthur T. White},
       DOI = {10.1002/mana.19881360120},
       URLX = {https://doi-org.proxy.uba.uva.nl/10.1002/mana.19881360120},
}

@article {Bon94,
    AUTHOR = {Bonnington, C. Paul},
     TITLE = {The relative maximum genus of a graph},
   JOURNAL = {J. Combin. Theory Ser. B},
  FJOURNAL = {Journal of Combinatorial Theory. Series B},
    VOLUME = {60},
      YEAR = {1994},
    NUMBER = {2},
     PAGES = {195--206},
      ISSN = {0095-8956},
   MRCLASS = {05C10},
  MRNUMBER = {1271269},
MRREVIEWER = {Jozef \v{S}ir\'{a}\v{n}},
       DOI = {10.1006/jctb.1994.1013},
       URLX = {https://doi-org.proxy.uba.uva.nl/10.1006/jctb.1994.1013},
}

@article {YJ22,
    AUTHOR = {Yan, Qi and Jin, Xian'an},
     TITLE = {{$A$}-trails of embedded graphs and twisted duals},
   JOURNAL = {Ars Math. Contemp.},
  FJOURNAL = {Ars Mathematica Contemporanea},
    VOLUME = {22},
      YEAR = {2022},
    NUMBER = {2},
     PAGES = {Paper No. 6 (16 pages)},
      ISSN = {1855-3966},
   MRCLASS = {05C10 (05C45 05C62)},
  MRNUMBER = {4449169},
       DOI = {10.26493/1855-3974.2053.c7b},
       URLX = {https://doi-org.proxy.uba.uva.nl/10.26493/1855-3974.2053.c7b},
}

@article {Jun78,
    AUTHOR = {Jungerman, Mark},
     TITLE = {A characterization of upper-embeddable graphs},
   JOURNAL = {Trans. Amer. Math. Soc.},
  FJOURNAL = {Transactions of the American Mathematical Society},
    VOLUME = {241},
      YEAR = {1978},
     PAGES = {401--406},
      ISSN = {0002-9947},
   MRCLASS = {05C10 (05C40)},
  MRNUMBER = {492309},
MRREVIEWER = {E. A. Nordhaus},
       DOI = {10.2307/1998852},
       URLX = {https://doi.org/10.2307/1998852},
}

@article {FGM88,
    AUTHOR = {Furst, Merrick L. and Gross, Jonathan L. and McGeoch, Lyle A.},
     TITLE = {Finding a maximum-genus graph imbedding},
   JOURNAL = {J. Assoc. Comput. Mach.},
  FJOURNAL = {Journal of the Association for Computing Machinery},
    VOLUME = {35},
      YEAR = {1988},
    NUMBER = {3},
     PAGES = {523--534},
      ISSN = {0004-5411},
   MRCLASS = {68Q20 (05C10 68Q25 68R10)},
  MRNUMBER = {963159},
       DOI = {10.1145/44483.44485},
       URLX = {https://doi-org.proxy.library.vanderbilt.edu/10.1145/44483.44485},
}

@article {FSW92,
    AUTHOR = {Fleischner, H. and Sabidussi, G. and Wenger, E.},
     TITLE = {Transforming {E}ulerian trails},
      NOTE = {Algebraic graph theory (Leibnitz, 1989)},
   JOURNAL = {Discrete Math.},
  FJOURNAL = {Discrete Mathematics},
    VOLUME = {109},
      YEAR = {1992},
    NUMBER = {1-3},
     PAGES = {103--116},
      ISSN = {0012-365X},
   MRCLASS = {05C45},
  MRNUMBER = {1192373},
MRREVIEWER = {Ortrud R. Oellermann},
       DOI = {10.1016/0012-365X(92)90281-J},
       URLX = {https://doi.org/10.1016/0012-365X(92)90281-J},
}

@book {BJG09,
    AUTHOR = {Bang-Jensen, J\o{}rgen and Gutin, Gregory},
     TITLE = {{D}igraphs: {T}heory, {A}lgorithms and {A}pplications},
    SERIES = {Springer Monographs in Mathematics},
   EDITION = {2nd edition},
 PUBLISHER = {Springer-Verlag London, Ltd., London},
      YEAR = {2009},
     PAGES = {xxii+795},
      ISBN = {978-1-84800-997-4},
   MRCLASS = {05-01 (05C20 68-01 68R10 68W05)},
  MRNUMBER = {2472389},
       DOI = {10.1007/978-1-84800-998-1},
       URLX = {https://doi.org/10.1007/978-1-84800-998-1},
}

@book {Sch03a,
    AUTHOR = {Schrijver, Alexander},
     TITLE = {Combinatorial {O}ptimization. {P}olyhedra and {E}fficiency. {V}ol.
              {A}},
    SERIES = {Algorithms and Combinatorics},
    VOLUME = {24},
      NOTE = {Paths, flows, matchings,
              Chapters 1--38},
 PUBLISHER = {Springer-Verlag, Berlin},
      YEAR = {2003},
     PAGES = {xxxviii+647},
      ISBN = {3-540-44389-4},
   MRCLASS = {90-02 (05-02 52B55 68Q25 68R10 90C27 90C35 90C57)},
  MRNUMBER = {1956924},
MRREVIEWER = {Alexander I. Barvinok},
}

@article {Neb81,
    AUTHOR = {Nebesk\'{y}, Ladislav},
     TITLE = {A new characterization of the maximum genus of a graph},
   JOURNAL = {Czechoslovak Math. J.},
  FJOURNAL = {Czechoslovak Mathematical Journal},
    VOLUME = {31(106)},
      YEAR = {1981},
    NUMBER = {4},
     PAGES = {604--613},
      ISSN = {0011-4642},
   MRCLASS = {05C10 (05C35)},
  MRNUMBER = {631605},
MRREVIEWER = {Arthur T. White},
}

@article {SN90,
    AUTHOR = {{\v S}koviera, Martin and Nedela, Roman},
     TITLE = {The maximum genus of a graph and doubly {E}ulerian trails},
   JOURNAL = {Boll. Un. Mat. Ital. B (7)},
  FJOURNAL = {Unione Matematica Italiana. Bollettino. B. Serie VII},
    VOLUME = {4},
      YEAR = {1990},
    NUMBER = {3},
     PAGES = {541--551},
   MRCLASS = {05C10 (05C45)},
  MRNUMBER = {1073633},
MRREVIEWER = {Arthur T. White},
}

@article {GI05,
    AUTHOR = {Geelen, James and Iwata, Satoru},
     TITLE = {Matroid matching via mixed skew-symmetric matrices},
   JOURNAL = {Combinatorica},
  FJOURNAL = {Combinatorica. An International Journal on Combinatorics and
              the Theory of Computing},
    VOLUME = {25},
      YEAR = {2005},
    NUMBER = {2},
     PAGES = {187--215},
      ISSN = {0209-9683},
   MRCLASS = {05C70 (05B35)},
  MRNUMBER = {2127610},
MRREVIEWER = {Winfried Hochst\"{a}ttler},
       DOI = {10.1007/s00493-005-0013-7},
       URLX = {https://doi-org.proxy.library.vanderbilt.edu/10.1007/s00493-005-0013-7},
}

@article {GIM03,
    AUTHOR = {Geelen, James F. and Iwata, Satoru and Murota, Kazuo},
     TITLE = {The linear delta-matroid parity problem},
   JOURNAL = {J. Combin. Theory Ser. B},
  FJOURNAL = {Journal of Combinatorial Theory. Series B},
    VOLUME = {88},
      YEAR = {2003},
    NUMBER = {2},
     PAGES = {377--398},
      ISSN = {0095-8956},
   MRCLASS = {05B35},
  MRNUMBER = {1983366},
MRREVIEWER = {Laura Bertani},
       DOI = {10.1016/S0095-8956(03)00039-X},
       URLX = {https://doi-org.proxy.library.vanderbilt.edu/10.1016/S0095-8956(03)00039-X},
}

@article {BJ00,
    AUTHOR = {Bouchet, Andr\'{e} and Jackson, Bill},
     TITLE = {Parity systems and the delta-matroid intersection problem},
   JOURNAL = {Electron. J. Combin.},
  FJOURNAL = {Electronic Journal of Combinatorics},
    VOLUME = {7},
      YEAR = {2000},
     PAGES = {Research Paper 14 (22 pages)},
   MRCLASS = {05B35},
  MRNUMBER = {1741336},
MRREVIEWER = {Walter Wenzel},
       DOI = {10.37236/1492},
       URLX = {https://doi-org.proxy.library.vanderbilt.edu/10.37236/1492},
}

@incollection {Glu77,
    MYNOTE = {Fixed MR's Cyrillic-Latin transliteration for author name},
    OLDAUTHOR = {Gluhov, A. D.},    
    AUTHOR = {Glukhov, A. D.},
     TITLE = {A refutation of {R}ingeisen's conjectures},
 BOOKTITLE = {{G}raph {T}heory ({R}ussian)},
     PAGES = {90--91},
 PUBLISHER = {Izdanie Inst. Mat. Akad. Nauk Ukrain. SSR, Kiev},
      YEAR = {1977},
   MRCLASS = {05C10},
  MRNUMBER = {531866},
}

@incollection {Yav73,
    MYNOTE = {Fixed MR's Cyrillic-Latin transliteration for author name},
    OLDAUTHOR = {Javors\cprime ki\u{\i}, E. B.},
    AUTHOR = {Javors'ki\u{\i}, E. B.},
     TITLE = {One-component embedding of graphs, and {$\varphi
              $}-transformations},
 BOOKTITLE = {{$\varphi $}-transformations of graphs ({U}krainian)},
     PAGES = {97--102, 362--363},
 PUBLISHER = {Vidannja Inst. Mat. Akad. Nauk Ukra\"{\i}n. RSR, Kiev},
      YEAR = {1973},
   MRCLASS = {05C10 (57A05)},
  MRNUMBER = {398874},
MRREVIEWER = {J. Bos\'{a}k},
}

@article {Tut48,
    AUTHOR = {Tutte, W. T.},
     TITLE = {The dissection of equilateral triangles into equilateral
              triangles},
   JOURNAL = {Proc. Cambridge Philos. Soc.},
  FJOURNAL = {Proceedings of the Cambridge Philosophical Society},
    VOLUME = {44},
      YEAR = {1948},
     PAGES = {463--482},
      ISSN = {0008-1981},
   MRCLASS = {48.0X},
  MRNUMBER = {27521},
MRREVIEWER = {P. Scherk},
       DOI = {10.1017/s030500410002449x},
       URLX =
{https://doi-org.proxy.library.vanderbilt.edu/10.1017/s030500410002449x},
}

@incollection {Tut75,
    AUTHOR = {Tutte, W. T.},
     TITLE = {Duality and trinity},
 BOOKTITLE = {Infinite and finite sets ({C}olloq., {K}eszthely, 1973;
              dedicated to {P}. {E}rd\H{o}s on his 60th birthday), {V}ols. {I},
              {II}, {III}},
    SERIES = {Colloq. Math. Soc. J\'{a}nos Bolyai, Vol. 10},
     PAGES = {1459--1472},
 PUBLISHER = {North-Holland, Amsterdam-London},
      YEAR = {1975},
   MRCLASS = {05C99},
  MRNUMBER = {401555},
MRREVIEWER = {Robert L. Wilson, Jr.},
}

@article {EE-M22,
    AUTHOR = {Ellingham, Mark N. and Ellis-Monaghan, Joanna A.},
     TITLE = {A catalog of enumeration formulas for bouquet and dipole
embeddings under symmetries},
   JOURNAL = {Symmetry},
    VOLUME = {14},
      YEAR = {2022},
     PAGES = {paper 1793 (39 pages)},
}

@incollection {Bou87r,
    AUTHOR = {Bouchet, A.},
     TITLE = {Representability of {$\Delta$}-matroids},
 BOOKTITLE = {Combinatorics ({E}ger, 1987)},
    SERIES = {Colloq. Math. Soc. J\'{a}nos Bolyai},
    VOLUME = {52},
     PAGES = {167--182},
 PUBLISHER = {North-Holland, Amsterdam},
      YEAR = {1988},
   MRCLASS = {05B35},
  MRNUMBER = {1221555},
MRREVIEWER = {Andr\'{a}s Recski},
}

@incollection {topotutte,
    AUTHOR = {Chmutov, S.},
     TITLE = {Topological extensions of the {T}utte polynomial},
 BOOKTITLE = {Handbook of the Tutte Polynomial and Related Topics},
    EDITOR = {Ellis-Monaghan, J. and Moffatt, I. },
     PAGES = {497-513},
 PUBLISHER = {CRC Press},
      YEAR = {2022},
   MRCLASS = {},
  MRNUMBER = {},
MRREVIEWER = {},
       DOI = {},
       URL = {},
}

@article {PTZ04,
    AUTHOR = {Pisanski, Toma\v{z} and Tucker, Thomas W. and {\v Z}itnik, Arjana},
     TITLE = {Straight-ahead walks in {E}ulerian graphs},
   JOURNAL = {Discrete Math.},
  FJOURNAL = {Discrete Mathematics},
    VOLUME = {281},
      YEAR = {2004},
    NUMBER = {1-3},
     PAGES = {237--246},
      ISSN = {0012-365X},
   MRCLASS = {05C10 (05C45)},
  MRNUMBER = {2047770},
MRREVIEWER = {R. Bruce Richter},
       DOI = {10.1016/j.disc.2003.09.011},
       URLX =
{https://doi-org.proxy.library.vanderbilt.edu/10.1016/j.disc.2003.09.011},
}

@article {AE-M22,
    AUTHOR = {Abrams, Lowell and Ellis-Monaghan, Joanna A.},
     TITLE = {New dualities from old: generating geometric, {P}etrie, and
              {W}ilson dualities and trialities of ribbon graphs},
   JOURNAL = {Combin. Probab. Comput.},
  FJOURNAL = {Combinatorics, Probability and Computing},
    VOLUME = {31},
      YEAR = {2022},
    NUMBER = {4},
     PAGES = {574--597},
      ISSN = {0963-5483,1469-2163},
   MRCLASS = {05C10 (05C25 57M15)},
  MRNUMBER = {4439773},
MRREVIEWER = {Zhaoxiang\ Li},
       DOI = {10.1017/s096354832100047x},
       URLX = {https://doi.org/10.1017/s096354832100047x},
}

\end{document}